\documentclass[
	english
]{scrartcl}

\usepackage[utf8]{inputenc}

\usepackage{amsmath,amsfonts,amsthm,amssymb}
\usepackage[colorlinks,citecolor=blue]{hyperref}
\usepackage{cancel}
\usepackage{relsize}
\usepackage{algorithmic}

\usepackage{changes}

\usepackage{enumitem}
\setlist[enumerate]{label=(\alph*)}
\usepackage{pdfrender}

\usepackage[figure]{hypcap}
\usepackage{graphicx}
\graphicspath{{./img/}}
\usepackage{subcaption} 

\usepackage{preamble}

\usepackage{marginnote}

\usepackage[capitalise,nameinlink]{cleveref}
\allowdisplaybreaks

\numberwithin{equation}{section}

\crefname{assumption}{Assumption}{Assumptions}
\Crefname{ALC@unique}{Step}{Steps}

\newcommand\norm[1]{\left\Vert#1\right\Vert}
\newcommand\nnorm[1]{\Vert#1\Vert}

\newcommand\innerprod[2]{\left\langle #1, #2\right\rangle}
\newcommand\ninnerprod[2]{\langle #1, #2\rangle}

\newcommand\N{\mathbb{N}}
\newcommand\R{\mathbb{R}}

\newcommand\B{\mathbb{B}}

\newcommand\barR{\overline{\mathbb R}}

\newcommand\tto{\rightrightarrows}

\newcommand{\gph}{\operatorname{gph}}
\newcommand{\epi}{\operatorname{epi}}

\newcommand{\xb}{\bar x}

\renewcommand{\d}{\mathrm d}

\newcommand{\proxprenormal}{\widetilde{N}^{\textup{p}}}
\newcommand{\proxnormal}{N^{\textup{p}}}
\newcommand{\regnormal}{\widehat{N}}
\newcommand{\normal}{N}

\newcommand{\proxpresub}{\widetilde{\partial}^{\textup{p}}}
\newcommand{\proxsub}{\partial^{\textup{p}}}

\newcommand{\regsub}{\widehat{\partial}}
\newcommand{\sub}{\partial}
\newcommand{\subgeo}{\partial_{\textup{g}}}
\newcommand{\proxpresubgeo}{\widetilde{\partial}^{\textup{p}}_{\textup{g}}}
\newcommand{\proxsubgeo}{\partial^{\textup{p}}_{\textup{g}}}

\DeclareMathAlphabet{\mathpzc}{OT1}{pzc}{m}{it}
\newcommand\oo{\mathpzc{o}}

\newtheorem{theorem}{Theorem}[section]
\newtheorem{lemma}[theorem]{Lemma}
\newtheorem{proposition}[theorem]{Proposition}

\newtheorem{corollary}[theorem]{Corollary}
\newtheorem{remark}[theorem]{Remark}
\newtheorem{definition}[theorem]{Definition}
\newtheorem{example}[theorem]{Example}

\AddToHook{env/lemma/begin}{\crefalias{theorem}{lemma}}
\AddToHook{env/proposition/begin}{\crefalias{theorem}{proposition}}
\AddToHook{env/algorithm/begin}{\crefalias{theorem}{algorithm}}
\AddToHook{env/assumption/begin}{\crefalias{theorem}{assumption}}
\AddToHook{env/remark/begin}{\crefalias{theorem}{remark}}
\AddToHook{env/example/begin}{\crefalias{theorem}{example}}
\AddToHook{env/corollary/begin}{\crefalias{theorem}{corollary}}
\AddToHook{env/definition/begin}{\crefalias{theorem}{definition}}
\AddToHook{env/setting/begin}{\crefalias{theorem}{setting}}
\AddToHook{env/conjecture/begin}{\crefalias{theorem}{conjecture}}

\makeatletter
\long\def\@firstoffiveparen#1#2#3#4#5{\textup{\tagform@{#1}}}
\def\eqref@nolink#1{\textup{\tagform@{\ref*{#1}}}}
\def\eqref@link#1{%
\Hy@safe@activestrue
\expandafter\@setref\csname r@#1\endcsname\@firstoffiveparen{#1}%
\Hy@safe@activesfalse
}
\protected\def\eqref{\@ifstar\eqref@nolink\eqref@link}
\makeatother

\newif\ifshowcomments
\showcommentstrue 

\definecolor{mygreen}{rgb}{0.0,0.7,0.0}
\definecolor{mybrown}{rgb}{0.5,0.5,0.0}

\usepackage{mathtools}

\begin{document}

\title{%
	On directional local minimality and directional optimality conditions in nonsmooth optimization
	}%
\author{%
	Timm Baake%
	\footnote{%
		Philipps-Universität Marburg,
		Department of Mathematics and Computer Science,
		35032 Marburg,
		Germany,
		\email{baake@uni-marburg.de},
		\orcid{0009-0000-6068-465X}%
		}%
	\and
	Patrick Mehlitz%
	\footnote{%
		Philipps-Universität Marburg,
		Department of Mathematics and Computer Science,
		35032 Marburg,
		Germany,
		\email{mehlitz@uni-marburg.de},
		\orcid{0000-0002-9355-850X}%
		}%
	}%

\publishers{}
\maketitle

\begin{abstract}
	This paper considers the unconstrained minimization of a lower semicontinuous function. 
	Exploiting first and second subderivatives, directional limiting subdifferentials,
	and directional proximal subdifferentials,
	necessary and sufficient first- and second-order optimality conditions
	are derived that build upon the recently introduced notion of 
	directional local minimality. 
	These results then also yield optimality conditions for conventional nondirectional local minimality
	which are stated in terms of so-called critical directions
	and variational objects depending on them.
	Illustrative examples show that the derived conditions allow for a finer analysis than
	classical nondirectional optimality conditions.
\end{abstract}


\begin{msc}	
	\mscLink{49J52}, \mscLink{49J53}, \mscLink{90C46}
\end{msc}

\section{Introduction}\label{sec:intro}

Nonsmooth optimization has become a major topic in mathematical programming
due to numerous underlying applications in data science and machine learning
where, usually, nondifferentiable regularizers are employed to promote certain
desirable behavior of solutions, like sparsity patterns or bounded variation.
This paper addresses some fundamental aspects of nonsmooth optimization, 
namely first- and second-order optimality conditions for the unconstrained minimization of
a merely lower semicontinuous (lsc for short) function.
Exemplary, these can be stated in terms of subderivatives on the primal side and subdifferentials
on the dual side, see \cite{RockafellarWets1998} for a comprehensive study
and diverse historical notes, and this is also the path we are following here.
A striking first-order necessary optimality condition is Fermat's rule
which claims that the origin belongs to the subdifferential (in a certain sense) of the function of interest
at each local minimizer.
In this paper, we are going to refine this insight.

It has been realized in \cite{Gfrerer2013} that local minimizers of rather general constrained optimization problems
satisfy stationarity conditions with respect to so-called critical directions if only a rather mild
metric subregularity assumption holds in the respective critical direction of interest.
The underlying exploited variational tools, so-called directional limiting subdifferentials and coderivatives,
allow for a finer analysis than their nondirectional counterparts, resulting in sharper conditions.
These observations marked the starting point for a line of research
that concerns the derivation of these so-called directional optimality conditions
for diverse mathematically challenging problem classes like
bilevel optimization problems, 
see e.g.\ \cite{BaiYe2022,BaiYeZeng2025},
or optimization problems with geometric constraints,
see e.g.\ \cite{BenkoMehlitz2024,GfrererYeZhou2022,KaemingMehlitz2026,OuyangYeZhang2025},
with a particular emphasis on problems with disjunctive constraints,
see e.g.\ 
\cite{BenkoGfrererYeZhangZhou2023,BenkoGfrererYeZhangZhou2026,ChenLiuDaiKoebis2025,Gfrerer2014,LahoussineElIdrissiElYahyaoui2026}.
Most of these conditions are based on the aforementioned directional limiting tools of variational analysis,
see \cite{BenkoGfrererOutrata2019,LongWangYang2017,NgocVanVan2026} for an overview and a comprehensive study of their calculus.
However, one may also work with the so-called directional proximal tools of variational analysis, 
see e.g.\ \cite{BenkoGfrererYeZhangZhou2023,BenkoGfrererYeZhangZhou2026,BenkoMehlitz2023},
which take local second-order information into account.
Let us note that it is possible to distill directional limiting versions of Fermat's rule
from \cite[Corollary~4.1]{BenkoMehlitz2024} and \cite[Theorem~7(ii)]{Gfrerer2013},
but these results require local Lipschitzness of the underlying function.
The Fermat-type results stated in \cite[Section~5]{GinchevMordukhovich2012}
are based on a different notation of a directional limiting subdifferential 
which is weaker (in the sense of set inclusion) than the one from \cite{Gfrerer2013}
as pointed out in \cite[Section~5]{LongWangYang2017}.

The present paper has been motivated by the recently coined notion of directional local minimality,
see \cite[Definition~3.1]{OuyangYeZhang2025}, which replaces the conventional neighborhood of the
reference point in the usual definition of local minimality by a directional one.
It is clear by definition that conventional local minimality is the same as directional local minimality
in all unit directions, see \cref{lem:relation}, but this rather obvious insight allows to address
necessary and sufficient conditions for local minimality direction-wise.
We elaborate on primal first- and second-order necessary and sufficient optimality conditions
for directional local minimizers in terms of first and second subderivatives,
see \cref{prop:primal_fo_NOC,prop:primal_fo_SOC,prop:primal_so_NOC,prop:primal_so_SOC},
which are also used to characterize directional versions of the first- and second-order growth conditions.
Furthermore, we are concerned with the derivation of dual Fermat-type
necessary optimality conditions for the characterization of directional local minimizers
based on the directional limiting and proximal subdifferentials,
a topic which, apart from the exceptions mentioned earlier, has not yet been addressed in the literature to the best of our knowledge.
The corresponding findings are stated in \cref{thm:dual_NOC_limiting,thm:unconstrained_necessary_prox}
and provide the main results of this paper. 
Via the aforementioned relationship between conventional and directional local minimality,
our insights also allow for the derivation of directional optimality conditions that address the
standard notion of local minimality.
Throughout, illustrative examples are included to accompany our theoretical investigations
and to show the benefits of incorporating directional information into optimality conditions.
We emphasize that our considerations are not related to
directed subdifferentials, see e.g.\ \cite{BaierFarkhiRoshchina2012a,BaierFarkhiRoshchina2012b}, 
or directional convexificators, see e.g.\ \cite{DempePilecka2015}.

The remainder of the paper is structured as follows.
In \cref{sec:notation}, we comment on the exploited notation as well as the underlying concepts 
from variational analysis and generalized differentiation.
\cref{sec:primal_conditions} is concerned with the derivation of primal first- and second-order
necessary and sufficient optimality conditions based on first and second subderivatives.
Dual Fermat-type necessary optimality conditions, 
which build on the directional limiting and proximal subdifferentials, 
are investigated in \cref{sec:dual_conditions}.
\cref{sec:conclusions} closes the paper, 
summarizing the obtained findings and drafting directions of future research.

\section{Notation and preliminaries}\label{sec:notation}

The notation in this paper is fairly standard and follows, e.g., the one in \cite{RockafellarWets1998}.

\subsection{Fundamental notation}

Throughout the paper, $\N$ will denote the set of natural numbers without $0$, 
and the extended real line will be represented by $\barR:=\R\cup\{-\infty,\infty\}$.
We equip $\R^n$ with the Euclidean norm $\|\cdot\|$ and the Euclidean inner product $\innerprod{ \cdot }{ \cdot }$.
Throughout the paper, $\mathbb S:=\{u\in\R^n\,|\,\norm{u}=1\}$ is the unit sphere.
Given $u\in\R^n$, $\{u\}^\bot:=\{x^*\in\R^n\,|\,\innerprod{x^*}{u}=0\}$ is the annihilator of $u$.
For simplicity of notation,
$\bar x+\Omega:=\Omega+\bar x:=\{\bar x+x\in\R^n\,|\,x\in\Omega\}$ is exploited 
whenever $\Omega \subset \R^n$ is a set and $\bar x \in \R^n$ is arbitrary.
For $\varepsilon > 0$ and $\bar x \in \R^n$, we refer to
$\B_\varepsilon(\bar x):=\{x\in\R^n\,|\,\norm{x-\bar x}\leq\varepsilon\}$ as the closed ball of radius $\varepsilon$ around $\bar x$.
For $\varepsilon,\rho>0$ and a direction $u \in \R^n$,
\[
	\mathbb B_{\varepsilon,\rho}(u)
	:=
	\bigl\{w\in\mathbb B_\varepsilon(0)\,\bigl|\,
		\bigl\Vert \norm{w}u-\norm{u}w\bigr\Vert\leq\rho\,\norm{u}\norm{w}
	\bigr\}
\]
is used to represent a directional neighborhood of $u$. 
We note that $\B_{\varepsilon,\rho}(0)=\B_\varepsilon(0)$
and $\bar x + \mathbb B_{\varepsilon,\rho}(u)\subset\mathbb B_\varepsilon(\bar x)$.
There will be no distinction between a sequence and its elements, i.e.,
some sequence $\{x_k\}_{k \in \N}\subset\R^n$ is simply represented by $x_k$.
We use $x_k\to\bar x$ to denote the convergence of $x_k$ to $\bar x\in\R^n$.
Given a scalar sequence $t_k$, we exploit $t_k\downarrow 0$ to express that $t_k$ converges to $0$ from above.
We use $\oo\colon [0, \infty) \to \R$ to denote a function with the property $\oo(t)/t\to 0$ as $t \downarrow 0$.
An lsc function $\varphi\colon\R^n\to\barR$ 
is called proper if $\varphi(x)>-\infty$ holds for all $x\in\R^n$ and $\varphi(\bar x)<\infty$ is true for some $\bar x\in\R^n$.
Whenever $\Omega\subset\R^n$ is a nonempty and closed set,
we use $\iota_\Omega\colon\R^n\to\barR$ to denote its indicator function 
which vanishes on $\Omega$ and is set to $\infty$ on $\R^n\setminus\Omega$.
Clearly, $\iota_\Omega$ is a proper lsc function.

\subsection{Tools from variational analysis}

Here, we first comment on primal tools, which are based on the well-known tangent cone,
before addressing dual tools, which build upon the proximal, regular, and limiting normal cones.

\paragraph{Tangent cones and subderivatives}

Let $\Omega \subset \R^n$ be a set, and fix $\xb \in \Omega$. 
We refer to
\begin{equation*}
	T_\Omega(\xb)
	:=
	\{
	u \in \R^n
	\,|\,
	\exists t_k\downarrow 0,u_k\to u\colon
	 \xb + t_k u_k \in \Omega\ \forall k \in \N
	\}
\end{equation*}
as the tangent cone to $\Omega$ at $\xb$.
It is well known that $T_\Omega(\bar x)$ is a closed cone.

Let $\varphi\colon \R^n\to \overline{\R}$ be an lsc function, 
and fix $\bar x\in\R^n$ such that $|\varphi(\bar x)|<\infty$.
The graphical derivative of $\varphi$ at $\bar x$ is defined as the set-valued mapping
$D\varphi(\xb)\colon \R^n \tto \R$ given by
\begin{equation*}
	\forall u\in\R^n\colon\quad
	D\varphi(\bar x)(u)
	:= 
	\{
	v \in \R^m \,|\,
	(u, v) \in T_{\gph \varphi}((\xb, \varphi(\bar x)))
	\},
\end{equation*}
where $\gph\varphi:=\{(x,\varphi(x))\in\R^n\times\R\,|\,x\in\R^n,\,|\varphi(x)|<\infty\}$ 
denotes the graph of $\varphi$.
The function $\d \varphi(\xb)\colon\R^n\to\barR$ defined by
\begin{equation*}
		\forall u\in\R^n\colon\quad
		\d \varphi(\xb)(u)
		:=
		\liminf_{t\downarrow 0, u'\to u}
		\frac{\varphi(\xb + t u') - \varphi(\xb)}{t}
\end{equation*}
is called the first subderivative of $\varphi$ at $\xb$.
Observe that $\epi\d\varphi(\bar x)=T_{\epi\varphi}((\bar x,\varphi(\bar x)))$
is valid, where
$\epi\varphi:=\{(x,\alpha)\in\R^n\times\R\,|\,\varphi(x)\leq\alpha\}$
is the so-called epigraph of $\varphi$.
This relation also shows that $\d\varphi(\bar x)$ is an lsc function
which is positive homogeneous.
Given $x^*\in\R^n$, the so-called second subderivative of $\varphi$ at $\xb$ for $x^*$ is the lsc function $\d^2 \varphi(\xb;x^*)\colon\R^n\to\barR$ defined by
\begin{equation*}
		\forall u\in\R^n\colon\quad
		\d^2 \varphi(\xb; x^*)(u)
		:=
		\liminf_{t\downarrow 0, u'\to u}
		\frac{\varphi(\xb + t u') - \varphi(\xb) - t\langle x^*, u'\rangle}{\frac{1}{2}t^2}
		.
\end{equation*}
\begin{lemma}\label{lem:fst_subder_in_graphical_der}
	Let $\varphi\colon \R^n \to \barR$ be an lsc function,
	and fix $\xb \in \R^n$ such that $|\varphi(\xb)| < \infty$
	as well as some direction $u \in \R^n$ such that $|\d\varphi(\bar x)(u)|<\infty$. 
	Then we have
	\begin{equation*}
		\d\varphi(\xb)(u) \in D\varphi(\xb)(u).
	\end{equation*}
\end{lemma}
\begin{proof}
	Let $t_k \downarrow 0$ and $u_k \to u$ be sequences such that 
	\[
		\frac{\varphi(\bar x + t_ku_k)-\varphi(\bar x)}{t_k} \to \d\varphi(\bar x)(u).
	\]
	Note that this implies $|\varphi(\bar x + t_ku_k)|<\infty$ for all $k\in\N$
	potentially after merely considering the tail of the sequence.
	For each $k\in\N$, define 
	\[
		\nu_k := \frac{\varphi(\xb + t_k u_k) - \varphi(\xb)}{t_k}.
	\] 
	Then we have $\nu_k \to \d\varphi(\xb)(u)$, and
	$\varphi(\xb) + t_k \nu_k = \varphi(\xb + t_k u_k)$ holds for all $k\in\N$. Hence,
	for each $k\in\N$,
	$(\xb, \varphi(\xb)) + t_k (u_k, \nu_k) \in \gph \varphi$ is obtained,
	so that
	$(u, \d\varphi(\xb)(u)) \in T_{\gph\varphi}((\xb, \varphi(\xb)))$ follows.
	By definition of the graphical derivative, the assertion is obtained.
\end{proof}

\paragraph{Normal cones and subdifferentials}

Let $\Omega \subset \R^n$ be a closed set, and fix $\xb \in \Omega$. We refer to
\begin{align*}
	\proxnormal_\Omega(\xb)
	&:=
	\{x^* \in \R^n \,|\, 
	\exists \varepsilon,\gamma > 0,
	\forall x \in \Omega \cap \B_\varepsilon(\xb)
	\colon 
	\langle x^*, x - \xb\rangle \leq \gamma \|x - \xb\|^2\},\\
	\regnormal_\Omega(\xb)
	&:=
	\{x^* \in \R^n\,|\,\forall x \in \Omega\colon \langle x^*, x - \xb\rangle \leq o(\|x - \xb\|)\},
	\\
	\normal_\Omega(\xb)
	&:=
	\{x^* \in \R^n\,|\,
	\exists x_k \to \xb, x^*_k \to x^*\colon 
	x_k\in\Omega,\,x_k^* \in \regnormal_\Omega(x_k)\ 
	\forall k \in \N\}
\end{align*}
as the proximal, regular, and limiting normal cone to $\Omega$ at $\xb$, respectively,
and all of them are closed cones.
We also note that $\proxnormal_\Omega(\bar x)$ and $\regnormal_\Omega(\bar x)$ are convex.
It is well known that the inclusions 
$\proxnormal_\Omega(\bar x)\subset\regnormal_\Omega(\bar x)\subset\normal_\Omega(\bar x)$
hold true.
For the purpose of completeness, we set
$\proxnormal_\Omega(\tilde x) := \regnormal_\Omega(\tilde x) := \normal_\Omega(\tilde x) := \emptyset$ 
whenever $\tilde x \not \in \Omega$.

Let $\varphi\colon \R^n\to \overline{\R}$ be an lsc function, 
and fix $\bar x\in\R^n$ such that $|\varphi(\bar x)|<\infty$.
We refer to
\begin{align*}
	\proxsub \varphi(\xb)
	&:=
	\left\{
	x^* \in \R^n\ 
	\middle|\ 
	\begin{aligned}	
	&\exists \varepsilon,\gamma > 0,
	\forall x \in \B_\varepsilon(\xb)\colon \\
	&\quad\varphi(x) \geq \varphi(\xb) + \langle x^*, x - \xb\rangle -\gamma \|x - \xb\|^2
	\end{aligned}
	\right\},\\
	\regsub \varphi(\xb)
	&:=
	\{x^* \in \R^n\,|\,\forall x \in \R^n\colon 
	\varphi(x) \geq \varphi(\xb) +\langle x^*, x - \xb\rangle - \oo(\|x - \xb\|)\},
	\\
	\sub \varphi(\xb)
	&:=
	\{x^* \in \R^n\,|\,
	\exists x_k \to \xb,x^*_k \to x^*\colon \varphi(x_k)\to\varphi(\xb),\  
	x_k^* \in \regsub \varphi(x_k)\ \forall k \in \N\}
\end{align*}
as the proximal, regular, and limiting subdifferential of $\varphi$ at $\xb$, respectively.
These subdifferentials possess the equivalent representations
\begin{align*}
	\proxsub\varphi(\bar x)
	&=
	\{x^*\in\R^n\,|\,(x^*,-1)\in\proxnormal_{\epi\varphi}((\bar x,\varphi(\bar x)))\},
	\\
	\regsub\varphi(\bar x)
	&=
	\{x^*\in\R^n\,|\,(x^*,-1)\in\regnormal_{\epi\varphi}((\bar x,\varphi(\bar x)))\},
	\\
	\sub\varphi(\bar x)
	&=
	\{x^*\in\R^n\,|\,(x^*,-1)\in\normal_{\epi\varphi}((\bar x,\varphi(\bar x)))\},
\end{align*}
see \cite[Theorem~8.9, Proposition~8.46(a)]{RockafellarWets1998},
so the relations $\proxsub \varphi(\xb) \subset \regsub \varphi(\xb) \subset \sub \varphi(\xb)$
follow from the respective ones for the proximal, regular, and limiting normal cones.
Whenever $\varphi$ is convex,
all three subdifferentials coincide 
and reduce to the subdifferential of convex analysis, see \cite[Proposition~8.12]{RockafellarWets1998}.

\begin{remark}\label{rem:prox_sub_diff_and_2nd_subder}
	Let $\varphi\colon\R^n\to\barR$ be an lsc function, 
	and fix $\bar x\in\R^n$ such that $|\varphi(\bar x)|<\infty$.
	Furthermore, pick $x^*\in\proxsub\varphi(\bar x)$.
	Then we have $\d\varphi(\bar x)(u)\geq \innerprod{x^*}{u}$ and $\d^2\varphi(\bar x;x^*)(u)>-\infty$
	for all $u\in\R^n$.
\end{remark}

\section{Directional local minimality and primal optimality conditions}\label{sec:primal_conditions}

This section is dedicated to the introduction of directional local minimality,
see \cref{sec:dir_loc_min},
as well as the derivation of primal first- and second-order necessary and sufficient
optimality conditions that address directional local minimality,
based on first and second subderivatives, see \cref{sec:fo_primal,sec:so_primal}.
We emphasize that the findings from \cref{sec:fo_primal,sec:so_primal}
are not new per se as closely related results can be found in
\cite[Propositions~3.99,~3.100]{BonnansShapiro2000}.
However, focusing on directional local minimality
throws some further light on these findings,
and in order to keep the presentation self-contained, 
we decided to state the straightforward proofs here.

\subsection{The concept of directional local minimality}\label{sec:dir_loc_min}

To start, let us state the definition of directional local minimality
which has been inspired by \cite[Definition~3.1]{OuyangYeZhang2025}
where a more specific setting has been addressed.

\begin{definition}\label{def:directional_local_minimality}
	Let $\varphi\colon\R^n\to\barR$ be an lsc function,
	and let $u\in\R^n$ be fixed.
	Some point $\bar x\in\R^n$ such that $|\varphi(\bar x)|<\infty$ is referred to as a local minimizer of $\varphi$
	in direction $u$ if there exist constants $\varepsilon,\rho>0$ such that
	\begin{equation}\label{eq:dir_loc_min}
		\forall x\in\xb+\mathbb B_{\varepsilon,\rho}(u)\colon\quad
		\varphi(x)\geq\varphi(\bar x)
	\end{equation}
	holds. If the stronger condition
	\[
		\forall x\in\xb+\mathbb B_{\varepsilon,\rho}(u)\colon\quad
		x\neq\bar x\quad\Longrightarrow\quad\varphi(x)>\varphi(\bar x)
	\]
	is valid, $\bar x$ is called a strict local minimizer of $\varphi$ in direction $u$.
\end{definition}

For $u:=0$ in \cref{def:directional_local_minimality}, 
the conventional notion of a (strict) local minimizer is recovered.
Furthermore, whenever $u\neq 0$ and $\bar x$ is a (strict) local minimizer of $\varphi$
in direction $u$, then $\bar x$ is a (strict) local minimizer of $\varphi$ in direction
$\alpha u$ for each $\alpha>0$.
Hence, w.l.o.g.\ we may consider merely directions $u\in\mathbb S$
when addressing (strict) directional local minimality.

In the lemma below, we relate the directional notion of local minimality 
from \cref{def:directional_local_minimality} with its commonly used nondirectional counterpart.

\begin{lemma}\label{lem:relation}
	Let $\varphi\colon\R^n\to\barR$ be an lsc function,
	and fix $\bar x\in\R^n$ such that $|\varphi(\bar x)|<\infty$.
	Then the point $\bar x$ is a (strict) local minimizer of $\varphi$
	if and only if, for each $u\in\mathbb S$, 
	it is a (strict) local minimizer of $\varphi$ in direction $u$.
\end{lemma}
\begin{proof}
	The implication $[\Longrightarrow]$ is obvious.
	Hence, let us prove the converse implication $[\Longleftarrow]$.
	Therefore, we assume that $\bar x$ is not a (strict) local minimizer of
	$\varphi$. Then we find a sequence $x_k\to\bar x$ such that $x_k\neq\bar x$
	and $\varphi(x_k)<\varphi(\bar x)$ ($\varphi(x_k)\leq\varphi(\bar x)$)
	hold for each $k\in\N$. Set
	\begin{equation}\label{eq:standard_transfer_directional_setting}
		\forall k\in\N\colon\quad
		t_k:=\nnorm{x_k-\bar x},\qquad u_k:=\frac{x_k-\bar x}{t_k}.
	\end{equation}
	Then $t_k\downarrow 0$ and, along a subsequence (without relabeling),
	$u_k\to u$ for some $u\in\mathbb S$.
	Noting that we have $x_k=\bar x+t_ku_k$ for each $k\in\N$,
	$\bar x$ cannot be a (strict) local minimizer of $\varphi$ in direction $u$.
\end{proof}

\subsection{First-order optimality conditions}\label{sec:fo_primal}

To start, we state primal first-order necessary
optimality conditions based on the notion of the first subderivative.

\begin{proposition}\label{prop:primal_fo_NOC}
	Let $\varphi\colon\R^n\to\barR$ be an lsc function,
	and fix $\bar x\in\R^n$ such that $|\varphi(\bar x)|<\infty$.
	Then the following assertions hold.
	\begin{enumerate}
		\item\label{item:primal_fo_NOC_dir}
		Pick $u\in\mathbb S$.
			If $\bar x$ is a local minimizer of $\varphi$ in direction $u$,
			then $\d\varphi(\bar x)(u)\geq 0$ is valid.
		\item\label{item:primal_fo_NOC_full} 
			If $\bar x$ is a local minimizer of $\varphi$,
			then $\d\varphi(\bar x)(u)\geq 0$ is valid for all $u\in\mathbb S$.
	\end{enumerate}
\end{proposition}
\begin{proof}
	For the proof of the first assertion, we pick constants $\varepsilon,\rho>0$
	such that \eqref{eq:dir_loc_min} is valid.
	Pick arbitrary sequences $t_k\downarrow 0$ and $u_k\to u$.
	Then, for large enough $k\in\N$, we have $t_ku_k\in\mathbb B_{\varepsilon,\rho}(u)$,
	so that \eqref{eq:dir_loc_min} yields
	\[
		\liminf\limits_{k\to\infty}\frac{\varphi(\bar x+t_ku_k)-\varphi(\bar x)}{t_k}\geq 0.
	\]
	Hence, $\d\varphi(\bar x)(u)\geq 0$ follows by definition of the first subderivative,
	and the first assertion has been shown.
	The second assertion follows from the first one 
	while keeping \cref{lem:relation} in mind.
\end{proof}

\begin{remark}\label{rem:primal_fo_NOC_full_zero}
	Let $\varphi\colon\R^n\to\barR$ be an lsc function,
	fix $\bar x\in\R^n$ such that $|\varphi(\bar x)|<\infty$,
	and let $\bar x$ be a local minimizer of $\varphi$.
	Using similar arguments like in the proof of 
	\cref{prop:primal_fo_NOC}, 
	we find $\d\varphi(\xb)(0) \geq 0$.
	Hence, $\d\varphi(\bar x)(u)\geq 0$ holds for all $u\in\R^n$. 
\end{remark}

Next, we are going to state first-order sufficient optimality conditions
which are based on the first subderivative.
In order to do so, we employ the notion of (directional) first-order growth.

\begin{definition}\label{def:fo_growth_condition}
	Let $\varphi\colon\R^n\to\barR$ be an lsc function,
	and fix $\bar x\in\R^n$ such that $|\varphi(\bar x)|<\infty$
	as well as some direction $u\in\R^n$.
	We say that the first-order growth condition holds for $\varphi$ in direction $u$ at $\bar x$ 
	whenever there exist constants $\varepsilon,\rho,\kappa>0$
	such that
	\begin{equation}\label{eq:fo_growth_condition}
		\forall x\in\xb+\mathbb B_{\varepsilon,\rho}(u)\colon\quad
		\varphi(x)\geq\varphi(\bar x) + \kappa\norm{x-\bar x}
	\end{equation}
	is valid.
	Whenever the above holds for $u:=0$ (which makes the parameter $\rho$ meaningless),
	we say that the first-order growth condition holds for $\varphi$ at $\bar x$.
\end{definition}

Similar to the concept of directional local minimality from \cref{def:directional_local_minimality},
we may restrict the directions in \cref{def:fo_growth_condition} to be chosen from $\mathbb S$
whenever the directional version of the first-order growth condition is employed.

Our first result provides a characterization of the directional first-order growth condition
from \cref{def:fo_growth_condition} in terms of the first subderivative.
\begin{lemma}\label{lem:char_fo_growth_condition}
	Let $\varphi\colon\R^n\to\barR$ be an lsc function,
	and fix $\bar x\in\R^n$ such that $|\varphi(\bar x)|<\infty$ 
	as well as some direction $u\in\mathbb S$.
	Then the first-order growth condition holds for $\varphi$ in direction $u$ at $\bar x$ 
	if and only if $\d\varphi(\bar x)(u)>0$ is valid.
\end{lemma}
\begin{proof}
	For the proof of the implication $[\Longrightarrow]$,
	we assume that there are constants $\varepsilon,\rho,\kappa>0$
	such that \eqref{eq:fo_growth_condition} holds.
	Next, we pick arbitrary sequences $t_k\downarrow 0$ and $u_k\to u$.
	Then, for large enough $k\in\N$, we have $t_ku_k\in\mathbb B_{\varepsilon,\rho}(u)$,
	so that \eqref{eq:fo_growth_condition} yields
	\[
		\varphi(\bar x+t_ku_k)\geq\varphi(\bar x) + \kappa\,t_k\norm{u_k}.
	\]
	Hence,
	\[
		\liminf\limits_{k\to\infty}\frac{\varphi(\bar x+t_ku_k)-\varphi(\bar x)}{t_k}
		\geq
		\kappa
	\]
	follows from $u\in\mathbb S$,
	and $\d\varphi(\bar x)(u)\geq\kappa>0$ is obtained.
	\\
	For the proof of the converse implication $[\Longleftarrow]$,
	let us suppose that the first-order growth condition does not hold for $\varphi$
	in direction $u$ at $\bar x$.
	Then we find sequences $t_k\downarrow 0$ and $u_k\to u$ such that
	\[
		\forall k\in\N\colon\quad
		\varphi(\bar x + t_ku_k) < \varphi(\bar x) + \frac1k\,t_k\norm{u_k},
	\]
	and
	\[
		\d\varphi(\bar x)(u)
		\leq
		\liminf\limits_{k\to\infty}\frac{\varphi(\bar x+t_ku_k)-\varphi(\bar x)}{t_k}
		\leq
		\liminf\limits_{k\to\infty}\frac{\norm{u_k}}{k}
		=
		0
	\]
	follows.
\end{proof}

The upcoming lemma characterizes the nondirectional first-order growth condition
in terms of the directional one.

\begin{lemma}\label{lem:fo_growth_condition}
	Let $\varphi\colon\R^n\to\barR$ be an lsc function,
	and fix $\bar x\in\R^n$ such that $|\varphi(\bar x)|<\infty$.
	Then the first-order growth condition holds for $\varphi$ at $\bar x$
	if and only if, for each $u\in\mathbb S$, 
	the first-order growth condition holds for $\varphi$ in direction $u$ at $\bar x$.
\end{lemma}
\begin{proof}
	The implication $[\Longrightarrow]$ is obvious.
	Hence, let us prove the converse implication $[\Longleftarrow]$.
	Therefore, we assume that the first-order growth condition does not hold
	for $\varphi$ at $\bar x$. 
	Then we find a sequence $x_k\to\bar x$ such that 
	\[
		\forall k\in\N\colon\quad
		\varphi(x_k)<\varphi(\bar x) + \frac1k\nnorm{x_k-\bar x}
	\]
	holds. Note that this yields $x_k\neq\bar x$ for all $k\in\N$. Let us define sequences $t_k$ and $u_k$ as in
	\eqref{eq:standard_transfer_directional_setting}.
	Then $t_k\downarrow 0$ and, along a subsequence (without relabeling),
	$u_k\to u$ for some $u\in\mathbb S$.
	Noting that we have $x_k=\bar x+t_ku_k$ for each $k\in\N$,
	we find
	\[
		\d\varphi(\bar x)(u)
		\leq
		\liminf\limits_{k\to\infty}\frac{\varphi(\bar x+t_ku_k)-\varphi(\bar x)}{t_k}
		\leq
		\liminf\limits_{k\to\infty}\frac{1}{k}
		=
		0,
	\]
	so that \cref{lem:char_fo_growth_condition} shows that the
	first-order growth condition for $\varphi$ in direction $u$ at $\bar x$ fails.
\end{proof}

Combining \cref{lem:char_fo_growth_condition,lem:fo_growth_condition},
we find the following primal first-order sufficient optimality conditions.

\begin{proposition}\label{prop:primal_fo_SOC}
	Let $\varphi\colon\R^n\to\barR$ be an lsc function,
	and fix $\bar x\in\R^n$ such that $|\varphi(\bar x)|<\infty$.
	Then the following assertions hold.
	\begin{enumerate}
		\item\label{item:primal_fo_SOC_dir} 
			Pick $u\in\mathbb S$.
			If $\d\varphi(\bar x)(u)>0$ holds,
			then the first-order growth condition holds for $\varphi$ in direction $u$ at $\bar x$.
			Particularly, $\bar x$ is a strict local minimizer of $\varphi$ in direction $u$.
		\item\label{item:primal_fo_SOC_full} 
			If $\d\varphi(\bar x)(u)>0$ holds for all $u\in\mathbb S$,
			then the first-order growth condition holds for $\varphi$ at $\bar x$.
			Particularly, $\bar x$ is a strict local minimizer of $\varphi$.
	\end{enumerate}
\end{proposition}

Let us note that \cref{lem:char_fo_growth_condition} and \cref{prop:primal_fo_SOC}
are closely related to \cite[Proposition~3.99]{BonnansShapiro2000}.
However, our findings refine this result as they incorporate directional information.

\subsection{Second-order optimality conditions}\label{sec:so_primal}

Next, we are going to discuss primal second-order optimality conditions
based on the second subderivative. 
Therefore, we introduce the so-called critical cone
whose definition is motivated by the findings of \cref{sec:fo_primal}.

\begin{definition}\label{def:critical_cone}
	Let $\varphi\colon\R^n\to\barR$ be an lsc function,
	and fix $\bar x\in\R^n$ such that $|\varphi(\bar x)|<\infty$.
	Then
	\[
		\mathcal C_\varphi(\bar x)
		:=
		\{u\in\R^n\,|\,\d\varphi(\bar x)(u)\leq 0\}
	\]
	is referred to as the critical cone of $\varphi$ at $\bar x$.
\end{definition}

Recalling that the subderivative is an lsc function which is positive homogeneous
as its epigraph is a cone,
the critical cone is a closed cone.

\begin{remark}\label{rem:critical_cone_zero}
	Let $\varphi\colon\R^n\to\barR$ be an lsc function,
	and fix $\bar x\in\R^n$ such that $|\varphi(\bar x)|<\infty$.
	\begin{enumerate}
		\item From the definition of the first subderivative we get that $0 \in \mathcal C_\varphi(\xb)$.
		\item\label{item:critical_cone_zero_alt_repr} 
			In order to check local minimality of $\bar x$ for $\varphi$,
			one may first apply \cref{prop:primal_fo_NOC}.
			If there is $\tilde u\in\mathbb S$ such that $\d\varphi(\bar x)(\tilde u)<0$,
			then $\bar x$ cannot be a local minimizer of $\varphi$.
			Hence, it makes sense to merely consider first-order stationary points, i.e.,
			points $\bar x$ such that $\d\varphi(\bar x)(u)\geq 0$ holds for all $u\in\mathbb S$.
			Then $\{u\in\R^n\,|\,\d\varphi(\bar x)(u)=0\}$ 
			could be exploited to play the role of the critical cone.
			However, throughout the paper,
			we stick to the version which has been defined in \cref{def:critical_cone}.
	\end{enumerate}
\end{remark}

By definition of the critical cone,
we obtain the following corollary of \cref{prop:primal_fo_SOC}.

\begin{corollary}\label{cor:primal_fo_SOC}
	Let $\varphi\colon\R^n\to\barR$ be an lsc function,
	and fix $\bar x\in\R^n$ such that $|\varphi(\bar x)|<\infty$.
	If $\mathcal C_\varphi(\bar x)\cap\mathbb S=\emptyset$ holds,
	then the first-order growth condition holds for $\varphi$ at $\bar x$.
	Particularly, $\bar x$ is a strict local minimizer of $\varphi$.
\end{corollary}

Let us now start to state primal second-order necessary optimality conditions.

\begin{proposition}\label{prop:primal_so_NOC}
	Let $\varphi\colon\R^n\to\barR$ be an lsc function,
	and fix $\bar x\in\R^n$ such that $|\varphi(\bar x)|<\infty$.
	Then the following assertions hold.
	\begin{enumerate}
		\item\label{item:primal_so_NOC_dir} Pick $u\in\mathcal C_\varphi(\bar x)\cap\mathbb S$.
			If $\bar x$ is a local minimizer of $\varphi$ in direction $u$,
			then $\d^2\varphi(\bar x;0)(u)\geq 0$ is valid.
		\item\label{item:primal_so_NOC_full} If $\bar x$ is a local minimizer of $\varphi$,
			then $\d^2\varphi(\bar x;0)(u)\geq 0$ is valid
			for each $u\in\mathcal C_\varphi(\bar x)\cap\mathbb S$.
	\end{enumerate}
\end{proposition}
\begin{proof}
	For the proof of the first assertion, we pick constants $\varepsilon,\rho>0$ such that \eqref{eq:dir_loc_min} is valid.
	Pick arbitrary sequences $t_k\downarrow 0$ and $u_k\to u$.
	Then, for large enough $k\in\N$, we have $t_ku_k\in\mathbb B_{\varepsilon,\rho}(u)$, so that \eqref{eq:dir_loc_min} yields
	\[
		\liminf\limits_{k\to\infty}\frac{\varphi(\bar x+t_ku_k)-\varphi(\bar x)}{\frac12t_k^2}\geq 0.
	\]
	Hence, $\d^2\varphi(\bar x;0)(u)\geq 0$ follows by definition of the second subderivative,
	and the first assertion has been shown.
	The second assertion follows from the first one while keeping \cref{lem:relation} in mind.
\end{proof}

Let us note that, given a local minimizer $\bar x\in\R^n$ of $\varphi$ in direction $u\in\mathbb S$
such that $|\varphi(\bar x)|<\infty$,
we know $\d\varphi(\bar x)(u)\geq 0$ from \cref{prop:primal_fo_NOC}.
If $\d\varphi(\bar x)(u)>0$ holds, then we immediately find $\d^2\varphi(\bar x;0)(u)=\infty$,
and the second subderivative does not contribute any relevant information.
That is why \cref{prop:primal_so_NOC} only considers directions $u\in\mathcal C_\varphi(\bar x)\cap\mathbb S$.

Next, we discuss second-order sufficient optimality conditions
which are based on the second subderivative. 
Therefore, we employ the concept of (directional) second-order growth.

\begin{definition}\label{def:so_growth_condition}
	Let $\varphi\colon\R^n\to\barR$ be an lsc function,
	and fix $\bar x\in\R^n$ such that $|\varphi(\bar x)|<\infty$
	as well as some direction $u\in\R^n$.
	We say that the second-order growth condition holds for $\varphi$ in direction $u$ at $\bar x$
	whenever there exist constants $\varepsilon,\rho,\kappa>0$ such that
	\begin{equation}\label{eq:so_growth_condition}
		\forall x\in\xb+\mathbb B_{\varepsilon,\rho}(u)\colon\quad
		\varphi(x)\geq\varphi(\bar x) + \kappa\norm{x-\bar x}^2
	\end{equation}
	is valid. Whenever the above holds for $u:=0$ (which makes the parameter $\rho$ meaningless),
	we say that the second-order growth condition holds for $\varphi$ at $\bar x$.
\end{definition}

Similar as above, we will restrict ourselves to directions of unit length when employing
the directional second-order growth condition from \cref{def:so_growth_condition}.
Again, we start our considerations by presenting a characterization
of the directional second-order growth condition via the second subderivative.

\begin{lemma}\label{lem:char_so_growth_condition}
	Let $\varphi\colon\R^n\to\barR$ be an lsc function,
	and fix $\bar x\in\R^n$ such that $|\varphi(\bar x)|<\infty$
	as well as some direction $u\in\mathbb S$.
	Then the second-order growth condition holds for $\varphi$ in direction $u$ at $\bar x$ if and only if
	$\d^2\varphi(\bar x;0)(u)>0$ is valid.
\end{lemma}
\begin{proof}
	For the proof of the implication $[\Longrightarrow]$, we assume that there are constants $\varepsilon,\rho,\kappa>0$
	such that \eqref{eq:so_growth_condition} holds.
	Next, we pick arbitrary sequences $t_k\downarrow 0$ and $u_k\to u$.
	Then, for large enough $k\in\N$, we have $t_ku_k\in\mathbb B_{\varepsilon,\rho}(u)$,
	so that \eqref{eq:so_growth_condition} yields
	\[
		\varphi(\bar x+t_ku_k)\geq\varphi(\bar x)+\kappa\,t_k^2\norm{u_k}^2.
	\]
	Hence, 
	\[
		\liminf\limits_{k\to\infty}\frac{\varphi(\bar x+t_ku_k)-\varphi(\bar x)}{\frac12t_k^2}
		\geq
		2\kappa
	\]
	follows from $u\in\mathbb S$, and $\d^2\varphi(\bar x;0)(u)\geq 2\kappa>0$ is obtained.
	\\
	For the proof of the converse implication $[\Longleftarrow]$, let us suppose that the second-order growth condition
	does not hold for $\varphi$ in direction $u$ at $\bar x$.
	Then we find sequences $t_k\downarrow 0$ and $u_k\to u$ such that
	\[
		\forall k\in\N\colon\quad
		\varphi(\bar x+t_ku_k) < \varphi(\bar x) + \frac1kt_k^2\nnorm{u_k}^2,
	\]
	and
	\[
		\d^2\varphi(\bar x;0)(u)
		\leq
		\liminf\limits_{k\to\infty}\frac{\varphi(\bar x + t_ku_k)-\varphi(\bar x)}{\frac12t_k^2}
		\leq
		\liminf\limits_{k\to\infty}\frac{2\nnorm{u_k}^2}{k}
		=
		0
	\]
	follows.
\end{proof}

Now, we are in position to characterize the nondirectional second-order growth condition in terms of the directional one.

\begin{lemma}\label{lem:so_growth_condition}
	Let $\varphi\colon\R^n\to\barR$ be an lsc function,
	and fix $\bar x\in\R^n$ such that $|\varphi(\bar x)|<\infty$.
	Then the second-order growth condition holds for $\varphi$ at $\bar x$
	if and only if, for each $u\in\mathbb S$, 
	the second-order growth condition holds for $\varphi$ in direction $u$ at $\bar x$.
\end{lemma}
\begin{proof}
	The implication $[\Longrightarrow]$ is obvious.
	Hence, let us prove the converse implication $[\Longleftarrow]$.
	Therefore, we assume that the second-order growth condition does not hold
	for $\varphi$ at $\bar x$. 
	Then we find a sequence $x_k\to\bar x$ such that 
	\[
		\forall k\in\N\colon\quad
		\varphi(x_k)<\varphi(\bar x) + \frac1k\nnorm{x_k-\bar x}^2
	\]
	holds. Note that this yields $x_k\neq\bar x$ for all $k\in\N$. Let us define sequences $t_k$ and $u_k$ as in
	\eqref{eq:standard_transfer_directional_setting}.
	Then $t_k\downarrow 0$ and, along a subsequence (without relabeling),
	$u_k\to u$ for some $u\in\mathbb S$.
	Noting that we have $x_k=\bar x+t_ku_k$ for each $k\in\N$,
	we find
	\[
		\d^2\varphi(\bar x;0)(u)
		\leq
		\liminf\limits_{k\to\infty}\frac{\varphi(\bar x+t_ku_k)-\varphi(\bar x)}{\frac12t_k^2}
		\leq
		\liminf\limits_{k\to\infty}\frac{2}{k}
		=
		0,
	\]
	so that \cref{lem:char_so_growth_condition} shows that the
	second-order growth condition for $\varphi$ in direction $u$ at $\bar x$ fails.
\end{proof}

Combining \cref{lem:char_so_growth_condition,lem:so_growth_condition},
we find the following primal second-order sufficient optimality conditions.

\begin{proposition}\label{prop:primal_so_SOC}
	Let $\varphi\colon\R^n\to\barR$ be an lsc function,
	and fix $\bar x\in\R^n$ such that $|\varphi(\bar x)|<\infty$.
	Then the following assertions hold.
	\begin{enumerate}
		\item\label{item:primal_so_SOC_dir} 
			Pick $u\in\mathcal C_\varphi(\bar x)\cap\mathbb S$.
			If $\d^2\varphi(\bar x;0)(u)>0$ holds,
			then the second-order growth condition holds for $\varphi$ in direction $u$ at $\bar x$.
			Particularly, $\bar x$ is a strict local minimizer of $\varphi$ in direction $u$.
		\item\label{item:primal_so_SOC_full} 
			If $\d^2\varphi(\bar x;0)(u)>0$ holds for all $u\in\mathcal C_\varphi(\bar x)\cap\mathbb S$,
			then the second-order growth condition holds for $\varphi$ at $\bar x$.
			Particularly, $\bar x$ is a strict local minimizer of $\varphi$.
	\end{enumerate}
\end{proposition}

Given some point $\bar x\in\R^n$ such that $|\varphi(\bar x)|<\infty$ and some direction $u\in\mathbb S$
which does not belong to $\mathcal C_\varphi(\bar x)$,
we know $\d\varphi(\bar x)(u)>0$, and \cref{lem:char_fo_growth_condition}
shows that the first-order growth condition holds for $\varphi$ in direction $u$ at $\bar x$,
which clearly implies that the second-order growth condition holds for $\varphi$ in direction $u$ at $\bar x$.
Hence, \cref{prop:primal_so_SOC} merely considers directions from $\mathcal C_\varphi(\bar x)\cap\mathbb S$.

We emphasize that \cref{lem:char_so_growth_condition} and \cref{prop:primal_so_SOC}
are closely related to \cite[Proposition~3.100]{BonnansShapiro2000}
which, however, does not take directional information into account.

\section{Dual optimality conditions}\label{sec:dual_conditions}

Here, we are concerned with the derivation of dual Fermat-type first-order necessary optimality conditions
via the directional limiting and proximal subdifferentials,
see, respectively, \cref{sec:dual_conditions_lim,sec:dual_conditions_prox}.
Furthermore, in \cref{sec:dual_conditions_reg}, we comment on the potential notion
of a directional regular subdifferential.

\subsection{Conditions in terms of limiting subgradients}\label{sec:dual_conditions_lim}

To start the subsection, let us recall the definition of the so-called directional limiting normal cone
which dates back to \cite[Definition~2.1]{Gfrerer2013} and \cite[Definition~2.3]{GinchevMordukhovich2011},
see \cite[Section~2]{Gfrerer2014} as well.
For a closed set $\Omega\subset\R^n$, some point $\bar x\in\Omega$, and a direction $u\in\R^n$,
we refer to
\[
	\normal_\Omega(\bar x;u)
	:=
	\left\{
		x^*\in\R^n\,\middle|\,
		\begin{aligned}
		&\exists t_k\downarrow 0,\,u_k\to u,\,x_k^*\to x^*\colon
		\\
		&\quad
		\bar x+t_ku_k\in\Omega,\,x_k^* \in \regnormal_\Omega(\bar x + t_ku_k)\ \forall k\in\N
		\end{aligned}
	\right\}
\]
as the limiting normal cone to $\Omega$ in direction $u$ at $\bar x$.
Note that $\normal_\Omega(\bar x;0)=\normal_\Omega(\bar x)$,
and $\normal_\Omega(\bar x;u)=\emptyset$ holds if $u\notin T_\Omega(\bar x)$.
For the purpose of completeness, we set $\normal_\Omega(\tilde x;u):=\emptyset$ for $\tilde x\notin\Omega$.

Next, we recall the definition of the (analytic) directional limiting subdifferential,
see \cite[Definition~6]{Gfrerer2013} and \cite[Definition~5.1]{LongWangYang2017}.
For an lsc function $\varphi\colon\R^n\to\overline\R$, a point $\bar x\in\R^n$ such that $|\varphi(\bar x)|<\infty$,
and $u \in \R^n$, we refer to
\begin{equation*}
	\sub \varphi(\xb; u)
	:=
	\left\{
	x^* \in \R^n\ 
	\middle|\ 
	\begin{aligned}	
	&\exists t_k \downarrow 0, 
	u_k\to u, 
	x^*_k\to x^*\colon\\
	&\quad\varphi(\xb + t_k u_k)\to\varphi(\xb),
	x^*_k \in \regsub \varphi(\xb + t_k u_k)\ \forall k \in \N
	\end{aligned}
	\right\}
\end{equation*}
as the directional limiting subdifferential of $\varphi$ in direction $u$ at $\bar x$.
By construction, we have $\sub\varphi(\bar x;u)\subset\sub\varphi(\bar x)$.
We note that a slightly different notion of a directional limiting subdifferential 
has been introduced in \cite[Definition~3.1]{GinchevMordukhovich2011},
but the latter did not prevail as it does not take convergence of function values
into account. Similarly, the related construction from \cite[Definition~3.1]{GinchevMordukhovich2012}
did not outlive the tools discussed here.
For additionally given $\nu\in\R$,
\[
	\subgeo \varphi(\bar x;(u,\nu))
	:=
	\left\{
		x^*\in\R^n\,\middle|\,
		(x^*,-1)\in\normal_{\epi\varphi}((\bar x,\varphi(\bar x));(u,\nu))
	\right\}
\]
is referred to as the geometric directional limiting subdifferential of $\varphi$ in direction $(u,\nu)$ at $\bar x$.
Note that the latter object is only reasonable if 
$(u,\nu)\in T_{\epi\varphi}((\bar x,\varphi(\bar x)))=\epi\d\varphi(\bar x)$,
i.e., if $\nu\geq\d\varphi(\bar x)(u)$.
A comprehensive overview of calculus rules for directional limiting normal cones and subdifferentials
can be found in \cite{BenkoGfrererOutrata2019,LongWangYang2017,NgocVanVan2026}.
Exemplary, \cite[Proposition~4.1]{BenkoGfrererOutrata2019} shows that we always have the inclusion
\begin{equation}\label{eq:relation_lim_normals_and_subgradients}
	\bigcup\limits_{\nu\in D\varphi(\bar x)(u)} \subgeo\varphi(\bar x;(u,\nu))
	\subset
	\sub \varphi(\bar x;u),
\end{equation}
and that equality holds in \eqref{eq:relation_lim_normals_and_subgradients} for all $u\in\mathbb S$.

Below, we state dual first-order necessary optimality conditions based on 
geometric directional limiting subgradients.
The presented proof has been inspired by the ones of
\cite[Theorem~3.18]{KaemingMehlitz2026} and \cite[Proposition~3.2]{OuyangYeZhang2025}.

\begin{theorem}\label{thm:dual_NOC_limiting}
	Let $\varphi\colon\R^n\to\barR$ be an lsc function,
	and fix $\bar x\in\R^n$ such that $|\varphi(\bar x)|<\infty$.
	Then the following assertions hold.
	\begin{enumerate}
		\item\label{item:dual_NOC_limiting_dir} 
			Pick $u\in\mathcal C_\varphi(\bar x)\cap\mathbb S$.
			If $\bar x$ is a local minimizer of $\varphi$ in direction $u$,
			then $0\in\subgeo\varphi(\bar x;(u,\d\varphi(\bar x)(u)))$ is valid.
		\item\label{item:dual_NOC_limiting_nondir} 
			If $\bar x$ is a local minimizer of $\varphi$,
			then $0\in\sub\varphi(\bar x)$,
			and $0\in\subgeo\varphi(\bar x;(u,\d\varphi(\bar x)(u)))$ is valid for each $u\in\mathcal C_\varphi(\bar x)\cap\mathbb S$.
	\end{enumerate}
\end{theorem}
\begin{proof}
	Let us start to prove assertion~\ref{item:dual_NOC_limiting_dir}.
	Therefore, we pick constants $\varepsilon,\rho>0$ such that \eqref{eq:dir_loc_min} holds.
	Consider the function $\widetilde \varphi:= \varphi + \iota_{\xb + \B_{\varepsilon,\rho}(u)}$
	which is proper, lsc, and bounded from below by $\varphi(\bar x)$.
	Furthermore, $\widetilde\varphi$ possesses the global minimizer $\bar x$, and
	\begin{equation}\label{eq:same_subdif}
		\forall x\in\xb + \B_{\varepsilon/2,\rho/2}(u)\colon\quad
		x\neq\bar x
		\quad\Longrightarrow\quad
		\regsub \widetilde \varphi(x) = \regsub \varphi(x)
	\end{equation} 
	holds trivially.
	From \cref{prop:primal_fo_NOC} and $u\in\mathcal C_\varphi(\bar x)$
	we know $\d\varphi(\bar x)(u)=0$.
	Hence, there exist sequences $t_k \downarrow 0$ and $u_k \rightarrow u$ with
	\begin{equation}\label{eq:subderivative_to_zero}
		\frac{\varphi(\xb + t_k u_k) - \varphi(\xb)}{t_k} \rightarrow 0,
	\end{equation}
	which implies $\varphi(\xb + t_k u_k) \rightarrow \varphi(\xb)$.
	Let us set $x_k:=\bar x + t_ku_k$ for each $k\in\N$.
	Then we have $\varphi(x_k)\to\varphi(\bar x)$ and $(x_k-\bar x)/t_k\to u$.
	By passing to a subsequence (without relabeling) if necessary,
	we may assume that $x_k \in \xb + \B_{\varepsilon,\rho}(u)$ 
	is valid for all $k \in \N$.
	Taking \eqref{eq:dir_loc_min} and \eqref{eq:subderivative_to_zero}
	into account, we find a sequence $r_k \downarrow 0$ such that
	\begin{equation}\label{eq:consequence_subderivative_to_zero}
		\forall k \in \N\colon\quad 
		0 < \varphi(x_k) - \varphi(\xb) + t_k^2 \leq t_kr_k^2.
	\end{equation} 
	For each $k \in \N$, $x_k$ is a $\varphi(x_k) - \varphi(\xb) + t_k^2$-minimizer 
	of $\widetilde \varphi$ due to
	\begin{align*}
		\widetilde \varphi(x_k) 
		= 
		\varphi(x_k)
		<
		\varphi(\xb) + \varphi(x_k) - \varphi(\xb) + t_k^2=
		\widetilde \varphi(\xb) + \varphi(x_k) - \varphi(\xb) + t_k^2.
	\end{align*}
	Applying \cite[Theorem~2.28]{Mordukhovich2006},
	which is a corollary of Ekeland's variational principle, yields,
	for each $k\in\N$, the existence of
	$y_k,\xi_k \in \R^n$  
	with
	\begin{subequations}\label{eq:ekeland}
	\begin{align}
		\label{eq:ekeland_i}
			\nnorm{y_k - x_k} 
			&\leq 
			t_k r_k,
			\\
		\label{eq:ekeland_ii}
			\widetilde \varphi (y_k) 
			&\leq 
			\widetilde \varphi(\xb) + \varphi(x_k) - \varphi(\xb) + t_k^2,
			\\
		\label{eq:ekeland_iii}
			\xi_k 
			&\in 
			\regsub \widetilde \varphi(y_k),
			\\
		\label{eq:ekeland_iv}
			\norm{\xi_k} 
			&\leq 
			\frac{\varphi(x_k) - \varphi(\xb) + t_k^2}{r_k t_k}.
	\end{align}
	\end{subequations}
	Employing \eqref{eq:ekeland_i}, we obtain the estimates
	\begin{equation*}
	\begin{aligned}
	\frac{\nnorm{y_k - \xb}}{t_k} 
	&\leq
	\frac{\nnorm{x_k - \xb}}{t_k} + \frac{\nnorm{y_k - x_k}}{t_k}
	\leq
	\frac{\nnorm{x_k - \xb}}{t_k} + r_k
	\rightarrow 
	\norm{u}
	=
	1,
	\\
	\frac{\nnorm{y_k - \xb}}{t_k} 
	&\geq
	\frac{\nnorm{x_k - \xb}}{t_k} - \frac{\nnorm{y_k - x_k}}{t_k}
	\geq
	\frac{\nnorm{x_k - \xb}}{t_k} - r_k
	\rightarrow 
	\norm{u}
	=
	1, 
	\end{aligned}
	\end{equation*}
	which yield $\nnorm{y_k -\xb}/t_k \rightarrow 1$ and,
	particularly, $y_k \rightarrow \xb$.	
	
	For each $k\in\N$, we set
	$t_k' := \nnorm{y_k - \xb}$ and 
	$u_k' := (y_k - \bar x)/t_k'$.
	Then we have $t'_k\downarrow 0$, $t'_k/t_k\to 1$,
	and we observe from \eqref{eq:ekeland_i} and $r_k\downarrow 0$ that
	\begin{align*}
		u'_k
		=
		\frac{y_k - \xb}{\nnorm{y_k - \xb}}
		&=
		\frac{x_k - \xb}{\nnorm{y_k - \xb}} 
		+
		\frac{y_k - x_k}{\nnorm{y_k - \xb}}
		=
		\frac{x_k - \xb}{t_k} \frac{t_k}{t_k'} 
		+
		\frac{y_k - x_k}{t_k} \frac{t_k}{t_k'}
		\to 
		u
	\end{align*}
	is valid.
	By construction, we have $y_k=\bar x + t'_ku'_k$ for each $k\in\N$.
	Additionally, due to $t'_k\downarrow 0$ and $u'_k\to u$,
	we may assume, by passing to a subsequence (without relabeling) if necessary,
	that $y_k\in\xb + \mathbb B_{\varepsilon/2,\rho/2}(u)$ holds for all $k\in\N$.
	Furthermore, $u\in\mathbb S$ guarantees $y_k\neq\bar x$ for all $k\in\N$.
	
	Exploiting \eqref{eq:dir_loc_min} and \eqref{eq:ekeland_ii}, we find
	\begin{equation*}
		\varphi(\xb) 
		\leq 
		\varphi(y_k) 
		= 
		\widetilde \varphi(y_k) 
		\leq 
		\varphi (x_k) + t_k^2 
		\to 
		\varphi(\xb),
	\end{equation*}
	which gives the convergences $\varphi(y_k)\to\varphi(\bar x)$ and, via \eqref{eq:subderivative_to_zero},
	\begin{align*}
		\left|\frac{\varphi(\bar x + t_k'u_k')-\varphi(\bar x)}{t_k'}\right|
		&\leq
		\left|\frac{\varphi(x_k) - \varphi(\bar x) + t_k^2}{t_k'}\right|
		\leq
		\left|\frac{\varphi(\bar x + t_ku_k)-\varphi(\bar x)}{t_k}\right|\frac{t_k}{t_k'} + t_k\frac{t_k}{t_k'}
		\to 0.
	\end{align*}
	For each $k\in\N$, let us set
	\[
		\nu'_k:=\frac{\varphi(\bar x + t_k'u_k')-\varphi(\bar x)}{t_k'},
	\]
	and observe from above that $\nu'_k \to 0 = \d\varphi(\bar x)(u)$ is valid.
	Using \eqref{eq:same_subdif} and \eqref{eq:ekeland_iii}, we find
	\begin{align*}
		\xi_k
		\in
		\regsub\varphi(y_k)
		&=
		\{x^*\in\R^n\,|\,(x^*,-1)\in\regnormal_{\epi\varphi}((y_k,\varphi(y_k)))\}
		\\
		&=
		\{x^*\in\R^n\,|\,(x^*,-1)\in\regnormal_{\epi\varphi}((\bar x,\varphi(\bar x))+t_k'(u_k',\nu_k'))\}
	\end{align*}
	for each $k\in\N$,
	and \eqref{eq:consequence_subderivative_to_zero} and \eqref{eq:ekeland_iv}
	show $\xi_k\to 0$.
	Hence, exploiting $(u_k',\nu_k')\to(u,\d\varphi(\bar x)(u))$, 
	we find $0 \in \subgeo \varphi(\xb; (u,\d\varphi(\bar x)(u)))$
	by definition of the geometric directional limiting subdifferential.
	
	Assertion~\ref{item:dual_NOC_limiting_nondir} is a trivial consequence of
	\cite[Theorem~10.1]{RockafellarWets1998}, \cref{lem:relation}, and 
	the already proven statement~\ref{item:dual_NOC_limiting_dir}.
\end{proof}

As a simple corollary of 
\cref{lem:fst_subder_in_graphical_der}, \cref{thm:dual_NOC_limiting}, and \eqref{eq:relation_lim_normals_and_subgradients},
we obtain the subsequently stated necessary optimality conditions 
in terms of the directional limiting subdifferential.

\begin{corollary}\label{cor:dual_NOC_limiting}
	Let $\varphi\colon\R^n\to\barR$ be an lsc function,
	and fix $\bar x\in\R^n$ such that $|\varphi(\bar x)|<\infty$.
	Then the following assertions hold.
	\begin{enumerate}
		\item\label{item:dual_NOC_limiting_dir_cor} 
			Pick $u\in\mathcal C_\varphi(\bar x)\cap\mathbb S$.
			If $\bar x$ is a local minimizer of $\varphi$ in direction $u$,
			then $0\in\sub\varphi(\bar x;u)$ is valid.
		\item\label{item:dual_NOC_limiting_nondir_cor} 
			If $\bar x$ is a local minimizer of $\varphi$,
			then $0\in\sub\varphi(\bar x)$,
			and $0\in\sub\varphi(\bar x;u)$ is valid for each $u\in\mathcal C_\varphi(\bar x)\cap\mathbb S$.
	\end{enumerate}
\end{corollary}

We note that, via the elementary sum rule for the directional limiting subdifferential,
see e.g.\ \cite[Theorem~5.6]{LongWangYang2017},
\cref{cor:dual_NOC_limiting}\,\ref{item:dual_NOC_limiting_dir_cor}
generalizes \cite[Proposition~3.2]{OuyangYeZhang2025}.
Whenever $\varphi$ is locally Lipschitz continuous at $\bar x$,
the assertion from \cref{cor:dual_NOC_limiting}\,\ref{item:dual_NOC_limiting_nondir_cor} 
can be distilled from 
\cite[Corollary~4.1]{BenkoMehlitz2024}
and
\cite[Theorem 7(ii)]{Gfrerer2013}.
Dual Fermat-type optimality conditions based on a weaker (in the sense of set inclusion) notion of the directional limiting subdifferential,
which did not prevail the more convenient tools from \cite{BenkoGfrererOutrata2019,Gfrerer2013}, 
can be found in \cite[Section~5]{GinchevMordukhovich2012}.

The following example illustrates that the conditions from \cref{thm:dual_NOC_limiting}\,\ref{item:dual_NOC_limiting_dir}
can, indeed, be sharper than those ones from \cref{cor:dual_NOC_limiting}\,\ref{item:dual_NOC_limiting_dir_cor}.
\begin{example}\label{ex:geometric_better}
	Consider the lsc function $\varphi\colon\R\to\R$ given by
	\[
		\forall x\in\R\colon\quad
		\varphi(x)
		:=
		\begin{cases}
			0				
				& \exists k\in\N\colon\,x\in(-(\frac12)^{3k-3},-3(\frac12)^{3k-1}),
			\\
			\frac{1}{4}x	 - 3(\frac12)^{3k+1}	
				& \exists k\in\N\colon\,x\in[-3(\frac12)^{3k-1},-(\frac12)^{3k-1}],
			\\
			x				
				& \text{otherwise,}
		\end{cases}
	\]
	as well as $\bar x:=0$ and the direction $u:=-1$.
	Observing that $\varphi(x)\geq x$ holds for all $x\in\R$,
	we find $\d\varphi(\bar x)(u)=-1$, i.e., $u\in\mathcal C_\varphi(\bar x)\cap\mathbb S$.
	Let us analyze the sequences given by
	$x_k := -(\frac12)^{3k-3}$, 
	$x'_k := -(\frac12)^{3k-1}$, and
	$x''_k := -3(\frac12)^{3k-1}$ for each $k \in \N$. 
	We emphasize that $\varphi$ is continuous at all points $x_k'$, $k\in\N$.
	Notice that, in $\gph\varphi$, $(x_k, \varphi(x_k))$ and $(x'_k, \varphi(x'_k))$ converge to $(\xb, 0)$ from direction $(u, \d\varphi(\xb)(u))$, 
	while $(x''_k, \varphi(x''_k))$ converges to $(\xb, 0)$ from direction $(u, -\frac12)$.
	Furthermore, the points $x_k$, $x'_k$, and $x''_k$, $k\in\N$,
	are all possible points where $\varphi$ is not differentiable, except for $\xb$ itself.
	For all $k \in \N$, we have
	\begin{align*}
		&\regnormal_{\epi\varphi}((x_k, \varphi(x_k)))
		=
		\{(r, -s)\in\R\times\R\,|\, s \geq 0,\ r \geq s\},\\
		&\regnormal_{\epi\varphi}((x'_k, \varphi(x'_k)))
		=
		\{(r, -s)\in\R\times\R\,|\, s\geq 0, \tfrac14 s \leq r \leq s\},\\
		&\regnormal_{\epi\varphi}((x''_k, \varphi(x''_k))) 
		= 
		\{(r,-s)\in\R\times\R\,|\, s \geq 0,\ r \leq \tfrac14 s\},
	\end{align*}
	which yields
	\begin{align*}
		\subgeo\varphi(\xb; (u, \d\varphi(\xb)(u))) = [\tfrac14, \infty),
		\qquad
		\subgeo\varphi(\xb; (u, -\tfrac12)) = (-\infty, \tfrac14],
	\end{align*} 
	and $\sub\varphi(\xb; u) = \R$ follows from \eqref{eq:relation_lim_normals_and_subgradients}.
	
	Hence, \cref{thm:dual_NOC_limiting}\,\ref{item:dual_NOC_limiting_dir} rules out 
	directional local minimality
	of $\bar x$ for $\varphi$ in direction $u$,
	while \cref{cor:dual_NOC_limiting}\,\ref{item:dual_NOC_limiting_dir_cor} does not.
	Of course, one can also use \cref{prop:primal_fo_NOC}\,\ref{item:primal_fo_NOC_dir} 
	to rule out $\bar x$
	as a local minimizer of $\varphi$ in direction $u$ due to $\d\varphi(\bar x)(u)<0$.
\end{example}

Given an lsc function $\varphi\colon\R^n\to\barR$, $\bar x\in\R^n$ such that $|\varphi(\bar x)|<\infty$,
and $u\in\mathcal C_\varphi(\bar x)\cap\mathbb S$ such that $\bar x$ is a local minimizer of $\varphi$ in direction $u$,
one may ask the question whether $0\in\subgeo\varphi(\bar x;(u,\nu))$ provides a necessary optimality condition
for $\nu\in D\varphi(\bar x)(u)\setminus\{\d\varphi(\bar x)(u)\}$, 
see \cref{lem:fst_subder_in_graphical_der} and \eqref{eq:relation_lim_normals_and_subgradients} as well.
As our next example shows, this is not the case in general.

\begin{example}\label{ex:geometric_NOC_via_general_nu}
	Consider the lsc function $\varphi\colon\R\to\R$ given by
	\[
		\forall x\in\R\colon\quad
		\varphi(x)
		:=
		\begin{cases}
			0	&	\text{if }x\in\{-1/k\,|\,k\in\N\}\cup[0,\infty),
			\\
			-x	&	\text{otherwise,}
		\end{cases}
	\]
	as well as $\bar x := 0$ and the direction $u:=-1$. Then we find $\d\varphi(\bar x)(u)=0$ and, thus, $-1\in\mathcal C_\varphi(\bar x)\cap\mathbb S$.
	Additionally, $\bar x$ is a local minimizer of $\varphi$ in direction $u$.
	Note that $D\varphi(\bar x)(u)=\{0,1\}$.
	We find $\subgeo\varphi(\bar x;(u,\d\varphi(\bar x)(u)))=\R$ and, thus, $0\in\subgeo\varphi(\bar x;(u,\d\varphi(\bar x)(u)))$  
	as announced in \cref{thm:dual_NOC_limiting}\,\ref{item:dual_NOC_limiting_dir},
	but for $1\in D\varphi(\bar x)(u)$,
	we obtain $\subgeo\varphi(\bar x;(u,1))=\{-1\}$, i.e., $0\notin \subgeo\varphi(\bar x;(u,1))$.
\end{example}

The following example shows that the refined directional necessary optimality condition
from \cref{cor:dual_NOC_limiting}\,\ref{item:dual_NOC_limiting_nondir_cor} may rule out
points which satisfy the nondirectional necessary optimality condition while not being local minimizers.

\begin{example}\label{ex:dual_NOC_limiting}
	Consider the lsc function $\varphi\colon\R\to\barR$ given by
	\[
		\forall x\in\R\colon\quad
		\varphi(x)
		:=
		\begin{cases}
			x		&	x\in(-\infty,0]\cup(1,\infty),\\
			\frac1k	&	k\in\N\colon\,x\in(1/(k+1),1/k],
		\end{cases}
	\]
	as well as $\bar x:=0$ which is not a local minimizer of $\varphi$. 
	Then we have $0\in\sub\varphi(\bar x)$
	by considering the sequence $x_k\to\bar x$ given by
	\[
		\forall k\in\N\colon\quad
		x_k:=\frac{1}{k+1} + \frac12\left(\frac{1}{k}-\frac{1}{k+1}\right),
	\]
	and observing that $\varphi(x_k)=\frac1k$ as well as $\regsub\varphi(x_k)=\{0\}$
	hold for each $k\in\N$.
	Hence, the classical nondirectional necessary condition from \cref{cor:dual_NOC_limiting}\,\ref{item:dual_NOC_limiting_nondir_cor}
	does not rule out $\bar x$ as a potential local minimizer of $\varphi$.
	However, as we have $\mathcal C_\varphi(\bar x)\cap\mathbb S=\{-1\}$,
	due to $\varphi(x)\geq x$ for all $x\in\R$,
	and $\partial\varphi(\bar x;-1)=\{1\}$,
	the directional necessary optimality condition from \cref{cor:dual_NOC_limiting}\,\ref{item:dual_NOC_limiting_nondir_cor} shows that $\bar x$
	cannot be a local minimizer of $\varphi$.
\end{example}

Let us, in the general situation, compare the nondirectional and directional necessary optimality conditions
stated in \cref{cor:dual_NOC_limiting}\,\ref{item:dual_NOC_limiting_nondir_cor}.

\begin{proposition}\label{prop:unconstrained_necessary_lim_relation}
	Let $\varphi\colon\R^n\to\barR$ be an lsc function,
	and fix $\bar x\in\R^n$ such that $|\varphi(\bar x)|<\infty$.
	We consider the following conditions.
	\begin{enumerate}
		\item\label{item:NOC_lim_undir}
			We have $0\in\sub\varphi(\bar x)$.
		\item\label{item:NOC_lim_dir}
			For each $u\in\mathcal C_\varphi(\bar x)\cap\mathbb S$, we have $0\in\sub\varphi(\bar x;u)$.
	\end{enumerate}
	Then we always have the implication [\ref{item:NOC_lim_dir}$\Longrightarrow$\ref{item:NOC_lim_undir}],
	and the converse implication [\ref{item:NOC_lim_undir}$\Longrightarrow$\ref{item:NOC_lim_dir}]
	is also true provided $\varphi$ is convex.
\end{proposition}
\begin{proof}
	For the proof of the generally valid implication
	[\ref{item:NOC_lim_dir}$\Longrightarrow$\ref{item:NOC_lim_undir}],
	we consider two cases.
	If, on the one hand, 
	$\mathcal C_\varphi(\bar x)\cap\mathbb S$ is empty, then \cref{cor:primal_fo_SOC}
	implies that $\bar x$ is a local minimizer of $\varphi$, 
	and $0\in\sub\varphi(\bar x)$ follows from 
	\cref{cor:dual_NOC_limiting}\,\ref{item:dual_NOC_limiting_nondir_cor}.
	If, on the other hand, there is some $u\in\mathcal C_\varphi(\bar x)\cap\mathbb S$,
	then $0\in\sub\varphi(\bar x;u)$ yields $0\in\sub\varphi(\bar x)$
	as we have $\sub\varphi(\bar x;u)\subset\sub\varphi(\bar x)$ by construction.
	
	 Next, we assume that $\varphi$ is convex, 
	 and prove validity of the implication [\ref{item:NOC_lim_undir}$\Longrightarrow$\ref{item:NOC_lim_dir}].
	 Hence, assume that $0\in\sub\varphi(\bar x)$.
	 By convexity of $\varphi$, this already implies that $\bar x$ is a minimizer of $\varphi$,
	 so that \cref{cor:dual_NOC_limiting}\,\ref{item:dual_NOC_limiting_nondir_cor} shows
	 $0\in\sub\varphi(\bar x;u)$ for all $u\in\mathcal C_\varphi(\bar x)\cap\mathbb S$.
\end{proof}

\cref{ex:dual_NOC_limiting} illustrates that the implication
[\ref{item:NOC_lim_undir}$\Longrightarrow$\ref{item:NOC_lim_dir}]
in \cref{prop:unconstrained_necessary_lim_relation} does, in general,
not hold if the function under consideration is not convex.
Furthermore, we emphasize that a result similar to \cref{prop:unconstrained_necessary_lim_relation}
can be stated and proven based on \cref{thm:dual_NOC_limiting}\,\ref{item:dual_NOC_limiting_nondir}
which exploits the geometric directional limiting subdifferential.

\subsection{Conditions in terms of proximal subgradients}\label{sec:dual_conditions_prox}

Let us start this subsection by stating the definition of the directional proximal normal cone
which dates back to \cite[Definition~2.8]{BenkoGfrererYeZhangZhou2023}.
For a closed set $\Omega\subset\R^n$, some point $\bar x\in\Omega$, 
and a direction $u\in T_\Omega(\bar x)$, we refer to
\[
	\proxprenormal_\Omega(\bar x;u)
	:=
	\{x^*\in\R^n\,|\,
		\exists\varepsilon,\rho,\gamma>0,\,\forall x\in\Omega\cap(\bar x + \mathbb B_{\varepsilon,\rho}(u))\colon\,
		\ninnerprod{x^*}{x-\bar x}\leq\gamma\nnorm{x-\bar x}^2
	\}
\]
as the directional proximal pre-normal cone to $\Omega$ in direction $u$ at $\bar x$.
It is clear from the definition that the inclusions
\begin{equation}\label{eq:estimates_proximal_prenormals}
	\{x^*\in\R^n\,|\,\innerprod{x^*}{u}<0\}
	\subset
	\proxprenormal_\Omega(\bar x;u)
	\subset
	\{x^*\in\R^n\,|\,\innerprod{x^*}{u}\leq 0\}
\end{equation}
hold,
which motivates the definition of
\[
	\proxnormal_\Omega(\bar x;u)
	:=
	\proxprenormal_\Omega(\bar x;u)\cap\{u\}^\perp,
\]
the so-called directional proximal normal cone to $\Omega$ in direction $u$ at $\bar x$.
In the case where $u'\notin T_\Omega(\bar x)$,
we set $\proxprenormal_\Omega(\bar x;u'):=\proxnormal_\Omega(\bar x;u'):=\emptyset$
for the purpose of completeness.
Similarly, $\proxprenormal_\Omega(\tilde x;\tilde u):=\proxnormal_\Omega(\tilde x;\tilde u):=\emptyset$
may holds for $\tilde x\notin\Omega$ and arbitrary $\tilde u\in\R^n$.

Observe that $\proxprenormal_\Omega(\bar x;0)=\proxnormal_\Omega(\bar x;0)=\proxnormal_\Omega(\bar x)$ is clearly valid by definition.
In \cite[Proposition 2.9]{BenkoGfrererYeZhangZhou2023}, it was shown that the inclusions
\begin{equation}\label{eq:normal_cone_inclusions}
	\proxnormal_\Omega(\xb) \cap \{u\}^\bot
	\subset
	\proxnormal_\Omega(\xb; u)
	\subset
	\regnormal_{T_\Omega(\xb)}(u)
	\subset
	N_{T_\Omega(\xb)}(u)
	\subset
	N_\Omega(\xb; u)
\end{equation}
hold for $u\in T_\Omega(\bar x)$, 
and that $\proxprenormal_\Omega(\xb; u)$ as well as $\proxnormal_\Omega(\xb; u)$ are convex. 

In \cite[Definition~3.7]{BenkoMehlitz2023}, the authors introduced a directional version
of the proximal subdifferential. Here, we recover its definition and comment on some
of its elementary properties.
Let $\varphi\colon \R^n\to\barR$ be an lsc function, 
fix $\xb \in \R^n$ such that $|\varphi(\xb)| < \infty$,
and pick
$u \in \R^n$ such that $|\d \varphi(\xb)(u)| < \infty$.
We refer to
\begin{equation*}
	\proxpresub \varphi(\xb; u) 
	:=
	\left\{
		x^* \in \R^n \, \middle|\,
		\begin{aligned} 
		&\exists \varepsilon,\rho,\gamma > 0,
		\forall x \in \xb + \B_{\varepsilon,\rho}(u)\colon\\
		&\quad\varphi(x) 
		\geq 
		\varphi(\xb) + \ninnerprod{x^*}{x - \xb} - \gamma \norm{x - \xb}^2
		\end{aligned}					
	\right\}		
	\end{equation*}
as the directional proximal pre-subdifferential of $\varphi$ in direction $u$ at $\bar x$.
Note that $\proxpresub \varphi(\xb; 0) = \proxsub \varphi(\xb)$.
Furthermore,
\begin{equation*}
	\proxsub\varphi(\bar x;u)
	:=
	\proxpresub\varphi(\bar x;u)
	\cap
	\{x^*\in\R^n\,|\,\d\varphi(\bar x)(u) = \innerprod{x^*}{u}\}
\end{equation*}
is referred to as the directional proximal subdifferential of $\varphi$ in direction $u$ at $\bar x$.

From the above definition we immediately find the inclusions
\begin{equation}\label{eq:inclusions_proxpresub}
	\{x^*\in\R^n\,|\,\d\varphi(\bar x)(u)>\innerprod{x^*}{u}\}
	\subset
	\proxpresub\varphi(\bar x;u)
	\subset
	\{x^*\in\R^n\,|\,\d\varphi(\bar x)(u)\geq\innerprod{x^*}{u}\}.
\end{equation}
Consequently, if $x^*\in\proxpresub\varphi(\bar x;u)$ satisfies
$\d\varphi(\bar x)(u)>\innerprod{x^*}{u}$,
it does not provide much information.
Hence, the definition of the directional proximal subdifferential is somewhat natural.
It has been shown in \cite[Lemma~3.8]{BenkoMehlitz2023} that
the directional proximal pre-subdifferential enjoys the equivalent representation
\begin{equation}\label{eq:alternative_rep_subproxpre}
	\proxpresub\varphi(\xb; u)
	=
	\{x^* \in \R^n \,|\, \d^2\varphi(\xb; x^*)(u) > -\infty\}
	.
\end{equation}

Some further comments on the directional proximal (pre-)subdifferential are collected in the following remark.

\begin{remark}\label{rem:dir_proxsub}
	Let $\varphi\colon \R^n\to\barR$ be an lsc function, 
	fix $\xb \in \R^n$ such that $|\varphi(\xb)| < \infty$,
	and pick
	$u \in \R^n$ such that $|\d \varphi(\xb)(u)| < \infty$.
	Then the following assertions hold.
	\begin{enumerate}
		\item\label{item:estimates_dir_proxsub} 
		We have the inclusions
		\begin{align*}
			\proxsub\varphi(\bar x)
			&\subset
			\proxpresub\varphi(\bar x;u),
			\\
			\proxsub\varphi(\bar x)\cap\{x^*\in\R^n\,|\,\d\varphi(\bar x)(u) = \innerprod{x^*}{u}\}
			&\subset
			\proxsub\varphi(\bar x;u).
		\end{align*}
		\item\label{item:zero_in_proxpresub_and_proxsub} 
		According to \eqref{eq:inclusions_proxpresub}, whenever $u\in\mathcal C_\varphi(\bar x)$ holds,
			conditions $0\in\proxpresub\varphi(\bar x;u)$ and $0\in\proxsub\varphi(\bar x;u)$ are equivalent.
		\item\label{item:zero_dir_proxsub_and_undir} 
		If $\bar x$ is a local minimizer of $\varphi$, then $\proxsub\varphi(\bar x;0)=\proxsub\varphi(\bar x)$,
		see \cref{rem:primal_fo_NOC_full_zero,rem:critical_cone_zero}.
	\end{enumerate}
\end{remark}

Let $\varphi\colon \R^n\to\barR$ be an lsc function, 
fix $\xb \in \R^n$ such that $|\varphi(\xb)| < \infty$,
and pick
$u \in \R^n$ such that $|\d \varphi(\xb)(u)| < \infty$.
Similar to the geometric directional limiting subdifferential, 
we can define the geometric directional proximal (pre-)subdifferential according to
\begin{align*}
	\proxpresubgeo\varphi(\xb; (u, \nu))
	&:=
	\{
	x^* \in \R^n
	\,|\,
	(x^*, -1) \in 
	\proxprenormal_{\epi\varphi}((\xb; \varphi(\xb)); (u, \nu))
	\},\\
	\proxsubgeo\varphi(\xb; (u, \nu))
	&:=
	\{
	x^* \in \R^n
	\,|\,
	(x^*, -1) \in
	\proxnormal_{\epi\varphi}((\xb; \varphi(\xb)); (u, \nu))
	\}
	,
\end{align*}
where $\nu\in\R$ is a given scalar.
Let us note that this definition is reasonable only if
$(u,\nu)\in T_{\epi\varphi}((\bar x,\varphi(\bar x)))=\epi\d\varphi(\bar x)$ holds,
i.e., $\nu\geq\d\varphi(\bar x)(u)$ must be valid.

Usage of \cite[Propositions~3.6(b), 3.10(b), and 3.13]{BenkoMehlitz2023} shows that, for each
$\nu \in D\varphi(\xb)(u)$, the inclusion
\begin{equation*}
	\proxpresub\varphi(\xb; u) \subset \proxpresubgeo\varphi(\xb; (u, \nu))
\end{equation*}
holds,
which in turn yields, via \cref{lem:fst_subder_in_graphical_der} and 
\cite[Corollary~3.14]{BenkoMehlitz2023},
\begin{equation}\label{eq:relation_pre_prox_subgradients}
	\proxpresub\varphi(\xb; u) = 
	\bigcap_{\nu \in D\varphi(\xb)(u)}
	\proxpresubgeo\varphi(\xb; (u, \nu))
	=
	\proxpresubgeo\varphi(\xb; (u, \d\varphi(\xb)(u)))
	.
\end{equation}
Furthermore, \cite[Corollary~3.14]{BenkoMehlitz2023} also reveals
\begin{equation}\label{eq:relation_prox_normals_and_subgradients}
		\proxsub\varphi(\bar x;u)
		=
		\proxsubgeo\varphi(\xb; (u,\d\varphi(\xb)(u))).
\end{equation}
This is clearly different from the much more complicated relationship
between the directional limiting subdifferential and its geometric counterpart.

Our next lemma shows that considering $\nu\in D\varphi(\bar x)(u)\setminus\{\d\varphi(\bar x)(u)\}$
in the definition of the geometric directional proximal subdifferential is not reasonable.
\begin{lemma}\label{lem:geometric_dir_prox_sub_trivial}
	Let $\varphi\colon \R^n \to \barR$ be an lsc function, 
	fix $\xb \in \R^n$ such that $|\varphi(\xb)| < \infty$, and 
	pick $u \in \R^n$ such that $|\d\varphi(\xb)(u)| < \infty$
	and $\nu\in(\d\varphi(\bar x)(u),\infty)$.
	Then we have $\proxsubgeo\varphi(\bar x;(u,\nu))=\emptyset$.
\end{lemma}
\begin{proof}
	By definition of the first subderivative,
	we have $(u,\nu)\in\{u\}\times[\d\varphi(\bar x)(u),\infty)\subset \epi\d\varphi(\bar x)=T_{\epi\varphi}((\bar x,\varphi(\bar x)))$,
	so that \eqref{eq:normal_cone_inclusions}, the definition of the regular normal cone,
	the product rule from \cite[Proposition~1.2]{Mordukhovich2006},
	and $\d\varphi(\bar x)(u)<\nu<\infty$ yield
	\begin{align*}
		\proxnormal_{\epi\varphi}((\bar x,\varphi(\bar x));(u,\nu))
		&\subset
		\regnormal_{T_{\epi\varphi}((\bar x,\varphi(\bar x)))}((u,\nu))
		\\
		&\subset 
		\regnormal_{\{u\}\times[\d\varphi(\bar x)(u),\infty)}((u,\nu))
		\\
		&=
		\regnormal_{\{u\}}(u)\times\regnormal_{[\d\varphi(\bar x)(u),\infty)}(\nu)
		=
		\R^n\times\{0\}.
	\end{align*}
	Thus, $\proxsubgeo\varphi(\bar x;(u,\nu))=\emptyset$ is obtained.
\end{proof}

In the lemma below, we investigate the relation between the (geometric)
directional proximal and limiting subdifferential.

\begin{lemma}\label{lem:prox_in_dir_lim_subdiff}
	Let $\varphi\colon \R^n \to \barR$ be an lsc function, 
	fix $\xb \in \R^n$ such that $|\varphi(\xb)| < \infty$, and 
	pick $u \in \R^n$ such that $|\d\varphi(\xb)(u)| < \infty$. 
	Then, for each $\nu\in D\varphi(\bar x)(u)$, we have
	\[
		\proxsubgeo \varphi(\bar x;(u,\nu))
		\subset
		\subgeo \varphi(\bar x;(u,\nu)),
	\]
	and 
	\begin{equation*}
		\proxsub \varphi(\xb; u) \subset \sub \varphi(\xb; u)
	\end{equation*}
	is valid as well.
\end{lemma}
\begin{proof}
	The first inclusion is a direct consequence of \eqref{eq:normal_cone_inclusions}
	and the definitions of the geometric directional proximal and limiting subdifferentials.
	For the proof of the second inclusion,
	we choose $\nu:=\d\varphi(\bar x)(u)$ in the first one,
	which is possible according to \cref{lem:fst_subder_in_graphical_der},
	and exploit \eqref{eq:relation_lim_normals_and_subgradients}
	as well as \eqref{eq:relation_prox_normals_and_subgradients}  
	to obtain
	\begin{align*}
		\proxsub \varphi(\xb; u) 
		=
		\proxsubgeo \varphi(\bar x;(u,\d\varphi(\bar x)(u)))
		\subset
		\subgeo\varphi(\bar x;(u,\d\varphi(\bar x)(u)))
		\subset
		\sub\varphi(\bar x;u).	
	\end{align*} 
	This completes the proof.
\end{proof}

The following result is now easily obtained from our preparations
and complements the directional necessary optimality conditions stated in \cref{cor:dual_NOC_limiting}.

\begin{theorem}\label{thm:unconstrained_necessary_prox}
	Let $\varphi\colon\R^n\to\barR$ be an lsc function,
	and fix $\bar x\in\R^n$ such that $|\varphi(\bar x)|<\infty$.
	Then the following assertions hold.
	\begin{enumerate}
	\item\label{item:unconstrained_necessary_prox_dir}
		Pick $u \in \mathcal C_\varphi(\xb) \cap \mathbb{S}$.
		If $\xb$ is a local minimizer of $\varphi$ in direction $u$, then $0 \in \proxsub \varphi(\xb; u)$.
	\item\label{item:unconstrained_necessary_prox_nondir}
		If $\xb$ is a local minimizer of $\varphi$, then
		$0 \in \proxsub \varphi(\xb)$, and 
		$0 \in \proxsub \varphi(\xb; u)$ is valid for each $u \in \mathcal C_\varphi(\xb) \cap \mathbb S$.
	\end{enumerate}
\end{theorem}
\begin{proof}
	For the proof of assertion~\ref{item:unconstrained_necessary_prox_dir},
	let $u \in\mathcal C_\varphi(\xb)\cap\mathbb S$ be arbitrarily chosen,
	and let $\xb$ be a local minimizer of $\varphi$ in direction $u$.
	From \cref{prop:primal_fo_NOC}\,\ref{item:primal_fo_NOC_dir} 
	and the definition of the critical cone we get 
	\begin{equation*}
		\d\varphi(\xb)(u) = 0 = \langle 0, u\rangle.
	\end{equation*}
	Furthermore, 
	\cref{prop:primal_so_NOC}\,\ref{item:primal_so_NOC_dir}
	yields $\d^2\varphi(\xb; 0)(u)\geq 0 > -\infty$.
	Applying \eqref{eq:alternative_rep_subproxpre}, hence, shows $0 \in \proxsub\varphi(\xb; u)$.
	
	For the proof of assertion~\ref{item:unconstrained_necessary_prox_nondir}, 
	let $\xb$ be local minimizer of $\varphi$.
	Then we find $\varepsilon>0$ such that $\varphi(x)\geq\varphi(\bar x)$ holds for all $x\in\mathbb B_\varepsilon(\bar x)$.
	Hence, for arbitrary $\gamma>0$,
	\begin{equation*}
		\forall x \in \B_\varepsilon(\xb)\colon\quad
		\varphi(x) \geq \varphi(\xb) \geq \varphi(\xb) - \gamma \|x - \xb\|^2
	\end{equation*}
	is true, and $0 \in \proxsub \varphi(\xb)$ follows.
	The second part of assertion~\ref{item:unconstrained_necessary_prox_nondir} immediately follows 
	from assertion~\ref{item:unconstrained_necessary_prox_dir} via \cref{lem:relation}.
\end{proof}

In the following remark, we discuss the value of potential necessary optimality conditions 
in terms of the geometric directional proximal (pre-)subdifferential.

\begin{remark}\label{rem:unconstrained_necessary_prox_geo}
	Let $\varphi\colon\R^n \to \barR$ be an lsc function, and
	fix $\xb \in \R^n$ such that $|\varphi(\xb)| < \infty$. 
	\begin{enumerate}
		\item\label{item:unconstrained_necessary_prox_geo_subder}
			Keeping \eqref{eq:relation_prox_normals_and_subgradients} in mind, 
			\cref{thm:unconstrained_necessary_prox} can be equivalently reformulated
			when replacing $\proxsub\varphi(\xb; u)$ by $\proxsubgeo\varphi(\xb; (u, \d\varphi(\xb)(u)))$,
			yielding a counterpart of \cref{thm:dual_NOC_limiting}.
			This is pretty much in contrast to \cref{sec:dual_conditions_lim}
			where the conditions from \cref{cor:dual_NOC_limiting}
			are, in general, weaker than those ones from \cref{thm:dual_NOC_limiting},
			see \cref{ex:geometric_better} again. 
		\item\label{item:unconstrained_necessary_prox_geo_other}
			Pick $u\in\mathcal C_\varphi(\bar x)\cap\mathbb S$ 
			such that $|\d\varphi(\bar x)(u)|<\infty$ holds.
			According to \cref{lem:geometric_dir_prox_sub_trivial},
			$\nu:=\d\varphi(\bar x)(u)$ is the only reasonable
			scalar in $D\varphi(\bar x)(u)$ such that
			$0\in\proxsubgeo\varphi(\bar x;(u,\nu))$ provides a meaningful condition.
		
		\item
			Pick $u \in \mathcal C_\varphi(\xb) \cap \mathbb S$
			and assume that $\bar x$ is a local minimizer of $\varphi$ in direction $u$.
			According to \cref{rem:dir_proxsub}\,\ref{item:zero_in_proxpresub_and_proxsub},
			\eqref{eq:relation_pre_prox_subgradients}, 
			and \cref{thm:unconstrained_necessary_prox}\,\ref{item:unconstrained_necessary_prox_dir},
			condition $0 \in \proxpresubgeo\varphi(\xb; (u, \nu))$
			provides a necessary optimality condition for each $\nu \in D\varphi(\xb)(u)$,
			and the most restrictive among them is obtained for $\nu:=\d\varphi(\bar x)(u)$
			which then just equals $0\in\proxsub\varphi(\bar x;u)$.
			
			Observe that \eqref{eq:estimates_proximal_prenormals} yields 
			\begin{equation*}
				\{x^* \in \R^n \,|\, \langle x^*, u \rangle < \nu\}
				\subset
				\proxpresubgeo\varphi(\xb; (u, \nu))
				\subset
				\{x^* \in \R^n \,|\, \langle x^*, u\rangle \leq \nu\}
			\end{equation*}
			for each $\nu\in D\varphi(\bar x)(u)$.
			In the case where $\nu\in D\varphi(\bar x)(u)\setminus\{\d\varphi(\bar x)(u)\}$
			is considered, we find $\nu>\d\varphi(\bar x)(u)$ and, thus,
			\begin{align*}
				\proxpresubgeo\varphi(\xb; (u,\d\varphi(\bar x)(u)))
				&\subset
				\{x^* \in \R^n \,|\, \innerprod{ x^*}{ u } \leq \d\varphi(\bar x)(u)\}
				\\
				&\subset
				\{x^*\in\R^n \,|\, \innerprod{x^*}{u} < \nu\}
				\subset
				\proxpresubgeo\varphi(\xb; (u, \nu)).
			\end{align*}
			Hence,
			the inclusion $0\in\proxpresubgeo\varphi(\bar x;(u,\nu))$
			does not provide a necessary condition of any reasonable strength
			as it is always valid whenever $\d\varphi(\bar x)(u)\geq 0$ holds,
			i.e., it does not provide more information 
			than \cref{prop:primal_fo_NOC}\,\ref{item:primal_fo_NOC_dir}.
	\end{enumerate}
\end{remark}
According to \cref{lem:prox_in_dir_lim_subdiff}, the necessary optimality conditions
stated in \cref{thm:unconstrained_necessary_prox} are sharper than those ones from
\cref{thm:dual_NOC_limiting} and \cref{cor:dual_NOC_limiting}. 
However, while the directional limiting tools have been shown to enjoy
full calculus, see \cite{BenkoGfrererOutrata2019,LongWangYang2017,NgocVanVan2026} again,
only some calculus rules for directional proximal normal cones and subdifferentials are known yet,
see \cite{BenkoGfrererYeZhangZhou2023,BenkoGfrererYeZhangZhou2026,BenkoMehlitz2023}.
In this regard, \cref{thm:dual_NOC_limiting} and \cref{cor:dual_NOC_limiting} may turn out to be much more practically useful
than \cref{thm:unconstrained_necessary_prox}, at least from today's point of view.
In fact, \cref{thm:unconstrained_necessary_prox} motivates a more detailed study on calculus rules
for directional proximal normal cones and subdifferentials.
In the nondirectional case, associated calculus rules for the proximal subdifferential can 
already be found in \cite{Bounkhel2016,MordukhovichNam2007}.
Typically, one is interested in optimality conditions for so-called composite optimization problems of type
\begin{equation}\label{eq:composite_problem}\tag{CP}
	\min\limits_x\quad f(x) + g(F(x))\quad\text{s.t.}\quad x\in\R^n
\end{equation}
where $f\colon\R^n\to\R$ and $F\colon\R^n\to\R^m$ are sufficiently smooth
and $g\colon\R^m\to\overline\R$ is proper and lsc.
Note that the composite framework covers the one of constrained optimization
as it is possible to choose $g$ to be the indicator function of a closed set.
In order to apply \cref{thm:dual_NOC_limiting,thm:unconstrained_necessary_prox} as well as \cref{cor:dual_NOC_limiting},
or the results from \cref{sec:primal_conditions},
to the situation at hand,
sum and chain rules are required for the underlying variational tools.
For (analytic and geometric) directional limiting subdifferentials as well as first and second subderivatives,
these can, indeed, be found in \cite{BenkoGfrererOutrata2019,LongWangYang2017} and \cite{BenkoMehlitz2023}, respectively.

Let us also note that, formally, \cref{thm:dual_NOC_limiting} can be seen as a corollary
of \cref{thm:unconstrained_necessary_prox} via \cref{lem:prox_in_dir_lim_subdiff}.
However, we decided to include a proof of \cref{thm:dual_NOC_limiting} 
to keep \cref{sec:dual_conditions_lim} self-contained.

In the following result, we compare the nondirectional and directional necessary optimality conditions
from \cref{thm:unconstrained_necessary_prox}\,\ref{item:unconstrained_necessary_prox_nondir}.
As it turns out, both conditions are equivalent.
This result is essentially different from the associated one for the limiting subdifferential,
see \cref{prop:unconstrained_necessary_lim_relation}.

\begin{proposition}\label{prop:unconstrained_necessary_prox_relation}
	Let $\varphi\colon\R^n\to\barR$ be an lsc function,
	and fix $\bar x\in\R^n$ such that $|\varphi(\bar x)|<\infty$.
	Then the following conditions are equivalent.
	\begin{enumerate}
		\item\label{item:NOC_prox_undir}
			We have $0\in\proxsub\varphi(\bar x)$.
		\item\label{item:NOC_prox_dir}
			For each $u\in\mathcal C_\varphi(\bar x)\cap\mathbb S$, we have $0\in\proxsub\varphi(\bar x;u)$.
	\end{enumerate}
\end{proposition}
\begin{proof}
	In order to prove implication [\ref{item:NOC_prox_undir}$\Longrightarrow$\ref{item:NOC_prox_dir}],
	we pick $u\in\mathcal C_\varphi(\bar x)\cap\mathbb S$
	arbitrarily.
	Then, according to \cref{rem:prox_sub_diff_and_2nd_subder}, 
	we find $\d\varphi(\bar x)(u)=0$, and 
	\cref{rem:dir_proxsub}\,\ref{item:estimates_dir_proxsub}
	already yields $0\in\proxsub\varphi(\bar x;u)$.
	
	Let us verify the converse implication 
	[\ref{item:NOC_prox_dir}$\Longrightarrow$\ref{item:NOC_prox_undir}].
	Hence, suppose that $0\notin\proxsub\varphi(\bar x)$ is valid,
	i.e., the condition stated in \ref{item:NOC_prox_undir} is violated.
	Then we find a sequence $x_k\to\bar x$ such that
	\[
		\forall k\in\N\colon\quad
		\varphi(x_k)<\varphi(\bar x) - k\nnorm{x_k-\bar x}^2.
	\]
	Note that this yields $x_k\neq\bar x$ for all $k\in\N$.
	Let us define sequences $t_k$ and $u_k$ as in \eqref{eq:standard_transfer_directional_setting}.
	Then $t_k\downarrow 0$ and, along a subsequence (without relabeling),
	$u_k\to u$ for some $u\in\mathbb S$.
	Furthermore, $x_k=\bar x+t_ku_k$ is valid for each $k\in\N$.
	Observe that
	\[
		\d\varphi(\bar x)(u)
		\leq
		\liminf\limits_{k\to\infty}\frac{\varphi(\bar x+t_ku_k)-\varphi(\bar x)}{t_k}
		\leq
		\liminf\limits_{k\to\infty}-k\,t_k
		\leq 0,
	\]
	yielding $u\in\mathcal C_\varphi(\bar x)\cap\mathbb S$.
	If $\d\varphi(\bar x)(u)=-\infty$, $0\notin\proxsub\varphi(\bar x;u)$
	follows by definition of the directional proximal subdifferential.
	If $\d\varphi(\bar x)(u)\in(-\infty,0]$,
	then taking into account \eqref{eq:alternative_rep_subproxpre}
	and 
	\[
		\d^2\varphi(\bar x;0)(u)
		\leq
		\liminf\limits_{k\to\infty}\frac{\varphi(\bar x+t_ku_k)-\varphi(\bar x)}{\frac12 t_k^2}
		\leq
		\liminf\limits_{k\to\infty}-2k
		=
		-\infty
	\]
	shows $0\notin\proxpresub\varphi(\bar x;u)$ 
	and, thus, $0\notin\proxsub\varphi(\bar x;u)$.
	Hence, the condition stated in \ref{item:NOC_prox_dir} is violated.
\end{proof}

The computation of $\proxsub\varphi(\bar x;u)$ for fixed $u\in\mathcal C_\varphi(\bar x)$
merely requires, via \eqref{eq:alternative_rep_subproxpre}, the calculation of $\d\varphi(\bar x)(u)$ and $\d^2\varphi(\bar x;\cdot)(u)$, 
while computing $\proxsub\varphi(\bar x)$ might be a laborious task
as this necessitates to inspect the full variational behavior of $\varphi$ locally around $\bar x$.
Particularly, the directional version of the necessary optimality condition
from \cref{thm:unconstrained_necessary_prox}\,\ref{item:unconstrained_necessary_prox_nondir}
may be used as a suboptimality condition so that the full computation of $\proxsub\varphi(\bar x)$
is not required. 

Our final example illustrates that suboptimality conditions deduced from \cref{thm:unconstrained_necessary_prox}
can be sharper than the ones resulting from \cref{thm:dual_NOC_limiting} and, thus, \cref{cor:dual_NOC_limiting}.

\begin{example}
	Consider the continuous function $\varphi\colon\R\to\R$ given by
	\begin{equation*}
		\forall x \in \R\colon\quad
		\varphi(x) := x \sqrt{|x|},
	\end{equation*}
	as well as $\xb := 0$ and the direction $u:=-1$. 
	Clearly, $\bar x$ is not a local minimizer of $\varphi$ in direction $u$.
	Since
	\begin{equation*}
		\d\varphi(\xb)(u) =
		\liminf_{t\downarrow 0, u'\to u}
		\frac{\varphi(tu')}{t}
		=
		\liminf_{t\downarrow 0,u'\to u}
		\sqrt t\,u'\sqrt{|u'|}
		= 
		0
	\end{equation*}
	holds, we have $u \in \mathcal C_\varphi(\bar x) \cap \mathbb{S}$.
	We note that $\varphi$ is continuously differentiable on $\R\setminus\{0\}$
	with derivative 
	$\varphi'(x) = \frac{3}{2}\sqrt{|x|}$ for all $x\in\R\setminus\{0\}$. 
	Set $t_k:=\frac1k$ and $x_k:=\bar x + t_ku=-\tfrac1k$ for each $k\in\N$.
	For each $k \in \N$, we find
	$\regsub \varphi(x_k) = \{\frac{3}{2}\frac{1}{\sqrt k}\}$, which yields 
	\begin{align*}
		(\tfrac32\tfrac{1}{\sqrt k},-1)
		&\in
		\regnormal_{\epi\varphi}((x_k,\varphi(x_k)))
		=
		\regnormal_{\epi\varphi}((\bar x,\varphi(\bar x)) + t_k(u,\nu_k)),
	\end{align*}
	where we used
	\[
		\nu_k:=\frac{\varphi(\bar x+t_ku)-\varphi(\bar x)}{t_k}.
	\]
	Observing that $\nu_k\to 0 = \d\varphi(\bar x)(u)$ holds,
	the above shows	$0 \in \subgeo \varphi(\xb; (u,\d\varphi(\bar x)(u)))$.
	However, for the second subderivative, we find
	\begin{equation*}
		\d^2\varphi(\bar x; 0)(u)
		\leq
		\liminf_{t\downarrow 0}
		-2 \frac{\sqrt t}{t}
		=
		-\infty,
	\end{equation*}
	which certificates $0 \not \in \proxsub \varphi(\xb; u)$ via \eqref{eq:alternative_rep_subproxpre}.
	Hence, \cref{thm:unconstrained_necessary_prox}\,\ref{item:unconstrained_necessary_prox_dir}
	indeed certificates that $\bar x$ is not a local minimizer of $\varphi$ in direction $u$,
	while \cref{thm:dual_NOC_limiting}\,\ref{item:dual_NOC_limiting_dir} and, thus,
	\cref{cor:dual_NOC_limiting}\,\ref{item:dual_NOC_limiting_dir_cor}
	do not.
\end{example}

\subsection{Conditions in terms of regular subgradients}\label{sec:dual_conditions_reg}

Following our argumentation from \cref{sec:dual_conditions_prox},
given an lsc function $\varphi\colon\R^n\to\overline{\R}$, a point $\bar x\in\R^n$ such that $|\varphi(\bar x)|<\infty$,
and some direction $u\in\R^n$ such that $|\d\varphi(\bar x)(u)|<\infty$,
one may refer to
\begin{equation}\label{eq:reg_presub}
	\left\{
		x^* \in \R^n		
		\ \middle|\ 
		\begin{aligned}		
		&\exists \varepsilon,\rho > 0,\forall x \in \xb + \B_{\varepsilon, \rho}(u)\colon\\
		&\quad\varphi(x) \geq \varphi(\xb) + \langle x^*, x - \xb\rangle - \oo(\|x - \xb\|)
		\end{aligned}
	\right\}
\end{equation}
as the directional regular pre-subdifferential of $\varphi$ in direction $u$ at $\bar x$.
For $u:=0$, this set corresponds to $\regsub\varphi(\bar x)$.
In the general case, a simple calculation reveals that the set in \eqref{eq:reg_presub} is the same as
\[
	\{x^*\in\R^n\,|\,\d\varphi(\bar x)(u)\geq\innerprod{x^*}{u}\},
\]
i.e., this construction delivers essentially the same information as the first subderivative $\d\varphi(\bar x)$.
Particularly, first-order necessary optimality conditions derived via this approach will turn out to be
equivalent to the ones in \cref{prop:primal_fo_NOC},
so a separate investigation is not required.
Particularly, we note that the above discussion yields
\[
	\regsub\varphi(\bar x)
	=
	\bigcap_{u\in\mathbb S}\{x^*\in\R^n\,|\,\d\varphi(\bar x)(u)\geq\innerprod{x^*}{u}\},
\]
and
\[
	0\in\regsub\varphi(\bar x)
	\quad\Longleftrightarrow\quad
	\forall u\in\mathbb S\colon\,\d\varphi(\bar x)(u)\geq 0
\]
follows, see \cite[Theorem~10.1]{RockafellarWets1998} as well.

\section{Conclusions}\label{sec:conclusions}

This paper has been devoted to the derivation of 
first- and second-order necessary and sufficient optimality conditions
for directional local minimizers associated with lsc functions.
Some special emphasis has been put on the investigation of Fermat-type conditions based on the
analytic and geometric directional proximal and limiting subdifferentials,
some recently introduced tools for the generalized differentiation of nonsmooth functions.
Simple examples have been presented in order to illustrate differences between 
the derived optimality conditions,
and these examples also underline that the derived directional conditions allow for a finer analysis
than their nondirectional counterparts.

In order to apply the findings of this paper to realistic scenarios in mathematical optimization,
like the composite model problem \eqref{eq:composite_problem},
sum and chain rules for the variational tools used in the aforementioned optimality conditions
have to be available.
On the one hand,
for subderivatives, these can, exemplary, be found in \cite{BenkoMehlitz2023},
and \cite{BenkoGfrererOutrata2019,LongWangYang2017,NgocVanVan2026} offer a comprehensive overview of calculus rules
for the (analytic and geometric) directional limiting subdifferential.
On the other hand, 
calculus rules for the (analytic and geometric) directional proximal subdifferential
are rarely available. Some first results can be found in \cite{BenkoMehlitz2023}.
Recalling that the obtained necessary optimality conditions in terms of directional proximal subgradients
are sharper than those ones based on directional limiting subgradients,
in our future research, 
we aim to study the calculus of directional proximal normals cones and subdifferentials in detail.
Some inspiration can be taken from \cite{Bounkhel2016,MordukhovichNam2007},
where the calculus of the nondirectional proximal subdifferential is addressed,
and \cite{BenkoGfrererYeZhangZhou2023,BenkoGfrererYeZhangZhou2026,BenkoMehlitz2023},
where some first calculus rules for the directional proximal normal cone can be found.
We may also take \cite{BenkoMehlitz2022} into account,
were a rather general approach to calculus has been promoted 
which applies to diverse variational objects.



\end{document}